\documentclass[10pt]{article}
\usepackage[left=2.5cm,right=2.5cm,top=2cm,bottom=2cm]{geometry}
\usepackage{latexsym}
\usepackage{cite}
\usepackage{enumitem}
\usepackage{amssymb,mathrsfs}
\usepackage{booktabs}
\usepackage{subfigure}
\usepackage{float}
\usepackage{color}
\usepackage{mathtools}
\usepackage{amsfonts}
\usepackage{amsthm}
\usepackage{setspace}
\usepackage{chemformula}
\usepackage{hyperref}
\usepackage{graphics}
\usepackage{graphicx}
\usepackage{caption}

\newtheorem{theorem}{Theorem}[section]
\newtheorem{definition}{Definition}[section]
\newtheorem{proposition}{Proposition}[section]

\newtheorem{lemma}{Lemma}[section]
\newtheorem{corollary}{Corollary}[section]
\newtheorem{remark}{Remark}[section]
\newtheorem{example}{Example}[section]

\begin{document}
\title{Optimality conditions and constraint qualifications for\\ cardinality constrained optimization problems}
\author{Zhuoyu Xiao\thanks{\baselineskip 9pt Department of Industrial and Manufacturing Engineering,  Pennsylvania State University, State College, PA, USA 168022. E-mail: zvx5148@psu.edu.} \ and \   Jane J. Ye\thanks{\baselineskip 9pt Corresponding author, Department of Mathematics
and Statistics, University of Victoria, Victoria, B.C., Canada V8W 2Y2. E-mail: janeye@uvic.ca.
The research of this author was partially supported by NSERC.}
}
\date{}

\maketitle
{\bf Abstract.} The cardinality constrained optimization problem (CCOP) is  an optimization problem where  the maximum number of nonzero components of any feasible point is bounded.
In this paper, 
we consider CCOP as a mathematical program with disjunctive subspaces constraints (MPDSC). Since a subspace is  
 a special case of a convex polyhedral set, MPDSC is a special case of
  the mathematical program with disjunctive constraints (MPDC).
 Using the  special structure of  subspaces, we are able to obtain more precise formulas for  the tangent and (directional) normal cones  for the  disjunctive set of subspaces.
We then obtain first and second order optimality conditions by using the corresponding results from MPDC.  Thanks to the  special structure of the subspace,   we are able to obtain some results for MPDSC that do not hold in general for MPDC. In particular we show that  the relaxed  constant positive linear dependence (RCPLD) is a sufficient condition for the metric subregularity/error bound property for MPDSC which is not true for MPDC in general. Finally we show that under all constraint qualifications presented in this paper,  certain exact penalization holds for CCOP.

{\bf Key Words.} cardinality constrained optimization problems,   disjunctive subspaces constraints, necessary optimality conditions, constraint qualifications, metric subregularity, error bounds property, RCPLD.

{\bf 2020 Mathematics Subject Classification.} 49J52, 49J53, 90C26, 90C46.

\section{Introduction}\label{section-1}
 In this paper, we consider the cardinality constrained optimization problem (CCOP) in the following form:
\begin{equation}\label{pro-1.1}
    \begin{aligned}
         \min_{x}\quad & f(x)\\
         \text{s.t.}\quad & g(x)\leq 0,\:h(x) = 0, \quad \|x\|_0\leq s,
    \end{aligned}
\end{equation}
where 
 $f:\mathbb{R}^{n}\to \mathbb{R}$, $g:\mathbb{R}^n\to \mathbb{R}^m$ and $h:\mathbb{R}^n\to \mathbb{R}^p$
and $\|x\|_{0}$ is the number of nonzero elements in the vector $x$ (also called $l_0$-norm). We assume $s<n$, otherwise the cardinality constraint would be superfluous.  Unless otherwise mentioned, we assume that all functions are smooth. Two simple examples about the cardinality constraint are given in Figure \ref{two examples}.

\begin{figure}[t]
    \begin{minipage}{0.5\linewidth}
        \centering
        \includegraphics[width=0.5\textwidth]{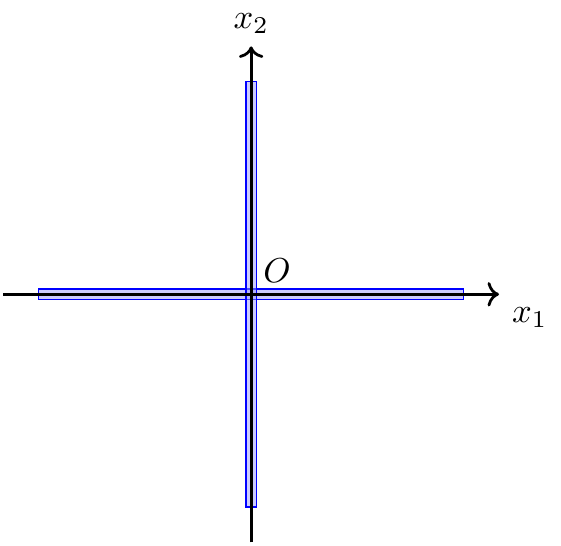}
    \end{minipage}
    \hfill
    \begin{minipage}{0.5\linewidth}
        \centering
        \includegraphics[width=0.5\textwidth]{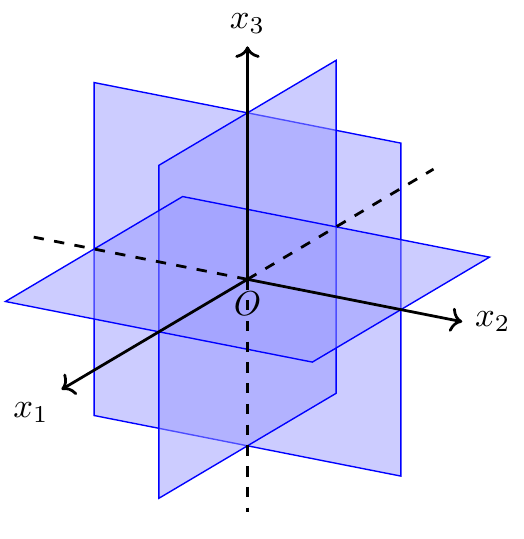}
    \end{minipage}
    \caption{Feasible set for the cardinality constraints: (i) $\|x\|_{0}\leq 1$ in $\mathbb{R}^2$; (ii) $\|x\|_{0}\leq 2$ in $\mathbb{R}^3$.}
    \label{two examples}
\end{figure}

 CCOPs have many   applications such as image processing, portfolio optimization, machine learning, and other related problems; see e.g.   \cite{Till-2021} and the references within for a survey of recent developments.
\par  CCOPs are intrinsically  nonconvex since even when all functions involved are convex, the feasible  region is still nonconvex. Moreover  problem  (\ref{pro-1.1}) can not be treated as a nonlinear program since $\|x\|_0$ is  a discontinuous function. 
For the case where $X:= \{x\in \mathbb{R}^{n}\:\vert\:g(x)\leq 0,\:h(x) = 0\}=\mathbb{R}^n$, Beck and Eldar  \cite{Beck-2013} introduced three concepts of stationarity conditions for CCOP, and proposed some efficient numerical algorithms aimed at finding these stationary points. The results are extended to include a so-called symmetric  set constraint in \cite{Beck-Hallak,Lu2015}.  Pan et al. \cite{pan2015} studied CCOP where  $X=\mathbb{R}^n$ by reformulating the problem as
$\min_x f(x) \quad \mbox{s.t. } x\in S,$
where $S:=\{x\in \mathbb{R}^n|\|x\|_0\leq s\}$.  They then give the expressions for the Bouligrand/Clarke tangent cone and the regular/Clarke normal cone to $S$ and present corresponding first and second order  optimality conditions. In \cite{Bauschke2014}, the formula for the limiting normal cone to $S$ is given. Recently Pan et al. \cite{pan2017optimality},  formulated  CCOP  as the following mathematical program:
\begin{equation}\label{pro-4.1}
    \begin{aligned}
    \min_{x} \quad & f(x)\\
    \textrm{s.t.} \quad & g(x)\leq 0, h(x) = 0, x\in S:= \bigcup\limits_{\mathcal{I}\in \mathcal{I}_{s}} \mathbb{R}_{\mathcal{I}},
    \end{aligned}
\end{equation}
 where $\mathcal{I}_{s} := \{\mathcal{I}\subseteq \{1,2,\dots,n\}\:\vert\:|\mathcal{I}|=s\}$ represents all subsets of $\{1,2,\dots,n\}$ with cardinality equal to $s$ and $\mathbb{R}_{\mathcal{I}} := \textrm{span} \{e_{i}\:\vert\:i\in \mathcal{I}\}$. By using some tools from variational analysis, they obtained some first and second order  optimality conditions and studied constraint qualifications for the necessary optimality conditions.
Note that each $\mathbb{R}_{\mathcal{I}}$ is a subspace and hence problem (\ref{pro-4.1}) is a special case of the 
mathematical program with disjunctive constraints (MPDC).
Recently  Burdakov et al. \cite{Burd-2016} introduced a relaxed complementarity-type reformulation of CCOP and proposed some numerical algorithms.
Based on this continuous reformulation,
some constraint qualifications and stationarity conditions for CCOP were introduced in \v{C}ervinka et al. \cite{Cerv-2016}. Some second-order necessary and sufficient optimality conditions for CCOP based on such reformulation are also introduced in Bucher and Kanzow  \cite{bucher2018second}. The relaxed complementarity-type reformulation is more computationally friendly than its original formulation. But  due to the extra variable introduced in the reformulation, the relaxed complementarity-type reformulation are not equivalent to the original problem in the sense of local optimality \cite{Burd-2016}.




\par In this paper, we study optimality conditions and constraint qualifications for problem  (\ref{pro-1.1}) by its equivalent reformulation (\ref{pro-4.1}). Recently there are many developments in the optimality conditions and constraint qualifications for MPDC \cite{Gfre-2014,Benk-2021,bai2019directional,Xu-2021}, see \cite{liang2021ye} and the references within for a survey. 

\par The main contributions of this paper are summarized as follows:
\begin{itemize}
   \item[(i)] We study the mathematical programs with disjunctive subspaces constraints (MPDSC) which include CCOP as well as the mathematical programs with switching constraints (MPSC) \cite{mehlitz2020stationarity}  as special cases. We explore the structure of the disjunctive set of subspaces to obtain formulas for various tangent and normal cones which do not hold if one of the subspaces is only a convex polyhedral set.
     Unlike in smooth systems with equality and inequality constraints, RCPLD may not imply the error bound property for the disjunctive system. Recently Xu and Ye \cite[Section 5.3]{Xu-2021} proposed RCPLD and the piecewise RCPLD for MPDCs, and showed that the  piecewise RCPLD is a sufficient condition for the error bound property.  Inspired by the fact that RCPLD coincides with the piecewise  RCPLD for MPSC \cite[Section 5.3]{Xu-2021}, we  show that this result actually holds for any system that can be formulated as a constraint system of  MPDSC.  
 Based on such result, we prove that for MPDSC, RCPLD is a sufficient condition for error bounds. 
\item[(ii)] Applying the results from  MPDSC, we obtain first and second order optimality conditions for CCOP and some new sufficient conditions for the error bound for CCOP. The  first order necessary and the second order necessary and sufficient optimality conditions we present  are sharper and  hold under weaker constraint qualifications than the corresponding results  in Pan et al. \cite{pan2017optimality}.  Moreover under all constraint qualifications presented in this paper, the exact penalty holds for CCOP. That is, if $x^*$ is a local minimizer of CCOP and one of the constraint qualifications discussed in this paper holds, then there exists a constant $\mu\geq 0$ such that $x^*$ is also a local minimizer of the exact penalty problem:
\begin{equation}
    \begin{aligned}
         \min_{x}\quad & f(x) +\mu\Bigl( \sum\limits_{i=1}^{m} \max\{g_{i}(x),0\} + \sum\limits_{i=1}^{p} |h_{i}(x)|\Bigr )\\
         \text{s.t.}\quad &  \|x\|_0\leq s.
    \end{aligned}\label{penalty}
\end{equation}
\end{itemize} 

\par The remainder of this paper is organized as follows. In Section \ref{section-2}, we derive formulas for various tangent and normal cones that  will  be used in this paper. In Section \ref{section-3}, we study optimality conditions and constraint qualifications for MPDSC. We also show that RCPLD is a sufficient condition for error bounds for MPDSC, an important result which does not hold for MPDC in general. In Section \ref{section-4}, we first reformulate CCOP as MPDSC, then apply the results in Section \ref{section-3} to CCOP. 

\section{Notation and preliminary results}\label{section-2}

\par The notations we adopt are standard. Given a point $x\in \mathbb{R}^n$,  $\mathbb{B}_{\varepsilon}(x)$ stands for the open ball of radius $\varepsilon$ centered at $x$, while the symbol $\mathbb{B}$ simply stands for the open unit ball centered at the origin. We denote by $\nabla f(x)$ the gradient of a continuously differentiable function $f:\mathbb{R}^n \to \mathbb{R}$ at $x$ and by $d_{\Omega}(x)$ the distance between the point $x$ and the set $\Omega$. Unless otherwise stated, $\|\cdot\|$ denotes an arbitrary norm in $\mathbb{R}^n$ and the notation $\langle\cdot, \cdot\rangle$ denotes the inner product. For a given nonempty set $A\subseteq \mathbb{R}^n$, we use notations cl$A$, cone$A$, and span$A$ to represent the closure of $A$, the conic hull of $A$, and the span of $A$, that is, the smallest subspace of $\mathbb{R}^n$ comprising $A$. For any given set $B$ with finite elements, we denote the number of elements in $B$ by $|B|$. Given   finite index sets $I, J$,  a pair  $(\{v_i\}_{i\in I}, \{u_i\}_{i\in J})$ of family of vectors $\{v_i\}_{i\in I}, \{u_i\}_{i\in J}$ is said to be positive linearly dependent if there exist scalars $\{\alpha_i\}_{i\in I}$ and $\{\beta_i\}_{i\in J}$ with $\alpha_i\geq 0$ for any $i\in I$, not all equal to zero such that $\sum_{i\in I}\alpha_i v_i+\sum_{i\in J}\beta_i u_i=0$. 

\par
The main purpose of this section is to derive some formulas for various tangent cones and normal cones to the disjunctive set of subspaces $S=\bigcup_{r=1}^{R}S_{r}$ where each $S_r$ is a subspace of $\mathbb{R}^n$. These formulas will be needed in the next sections of the paper.

\subsection{Formulas for tangent and normal cones}

Given a closed set $\Omega\subseteq \mathbb{R}^n$ and $x\in \Omega$, we denote the Bouligand tangent cone, the Fr\'echet/regular normal cone and the Mordukhovich/limiting normal cone to $\Omega$ at $x$ by $T_{\Omega}(x)$, $\hat{N}_{\Omega}(x)$ and $N_{\Omega}(x)$, respectively. When $\Omega$ is convex, we denote by  ${N}_{\Omega}(x)$  the normal cone in the sense of convex analysis. We  refer the reader to the standard reference of variational analysis in \cite{Bonn-2000, Clar-1990, Mord-2006, Mord-2018, Rock-2009} for their precise definitions.

\begin{proposition}\label{thm-3.1}
    Let $x^{\ast}\in S$. 
    Then 
    \begin{align}
        T_{S}(x^{\ast}) &= \bigcup\limits_{r\in I(x^{\ast})} T_{S_{r}}(x^{\ast}) = \bigcup\limits_{r\in I(x^{\ast})} S_{r}, \label{eqn-3.4}\\
        \hat{N}_{S}(x^{\ast}) &= \bigcap\limits_{r\in I(x^{\ast})} {N}_{S_{r}}(x^{\ast}) = \bigcap\limits_{r\in I(x^{\ast})} S_{r}^{\perp}, \label{eqn-3.5}\\
        N_{S}(x^{\ast}) &= \bigcup\limits_{r\in I(x^{\ast})} N_{S_{r}}(x^{\ast}) = \bigcup\limits_{r\in I(x^{\ast})} S_{r}^{\perp}, \label{eqn-3.6}
    \end{align}
    respectively, where  $I(x^{\ast}) := \{ r\in \{1,\dots,R\} \:\vert\: x^{\ast}\in S_{r} \}$ is the active index set for $S$ at $x^{\ast}$.
\end{proposition}

\begin{proof}
The first equation in (\ref{eqn-3.4}) is from \cite[Table 4.1]{Aubin} and the second equation in (\ref{eqn-3.4}) is obvious since each $S_{r}$ is a subspace and hence $T_{S_r}(x^*)=S_r$ provided that $x^*\in S_r$. The proof of (\ref{eqn-3.5}) follows from the one for the tangent cone by polarization.

Now it remains to prove that (\ref{eqn-3.6}) holds.
The second equation in (\ref{eqn-3.6}) is obvious from the fact that each $S_{r}$ is a subspace for $r = 1,\dots,R$. We only need to prove the first equation.
The inclusion $N_{S}(x^{\ast}) \subseteq \bigcup\limits_{r\in I(x^{\ast})} N_{S_{r}}(x^{\ast})$ follows from  \cite[Lemma 2.2]{Mehl-2020-O}. Now we prove the converse inclusion. 
First  from (13) and (14) in Adam et al. \cite{adam2016normally}, there exists $\delta>0$ such that
\begin{equation}\label{eqn-3.9}
   N_{S}(x^{\ast})=   \bigcup\limits_{x\in \mathbb{B}_{\delta}(x^{\ast})} \hat{N}_{S}(x).
\end{equation}
For each $r\in I(x^{\ast})$, since there are only finitely many subspaces we can pick a unit vector $u_{r}$ such that $u_{r}\in S_{r}$ but $u_{r}\notin S_{j}$ for other $j\neq r$. Consider the point $x=x^{\ast}+tu_{r}$ such that $t<\delta$. Then $x\in \mathbb{B}_{\delta}(x^{\ast})$ and  $x\in S_{r}$ but $x\notin S_{j}$ for other $j\neq r$. In this way, for every $i\in I(x^{\ast})$ we have
\begin{equation}
   {N}_{S_{r}}(x^{\ast}) = S_{r}^{\perp} = {N}_{S_{r}}(x) = \hat{N}_{S}(x)\subseteq \bigcup\limits_{x\in \mathbb{B}_{\delta}(x^{\ast})} \hat{N}_{S}(x),\label{eqn-3.8}
\end{equation}
where the third equality follows from the first equation in (\ref{eqn-3.5}) taking into account that $I(x)=\{r\}$. 
Combing  (\ref{eqn-3.9}) and (\ref{eqn-3.8}), we obtain the inclusion $\bigcup\limits_{r\in I(x^{\ast})} N_{S_{r}}(x^{\ast}) \subseteq N_{S}(x^{\ast})$ as desired, which completes the proof.
\end{proof}

In the following proposition, we will compute the regular normal cone to tangent cone $T_{S}(x^{\ast})$ at a tangent direction $d$.
\begin{proposition}\label{thm-3.5}
    \emph{(The regular normal cone to the tangent cone)} Let $x^{\ast}\in S$ and $d\in T_{S}(x^{\ast})$. Then the regular normal cone to $T_{S}(x^{\ast})$ at $d$ is calculated as
    \begin{equation}
        \hat{N}_{T_{S}(x^{\ast})}(d) = \bigcap\limits_{r\in I(x^{\ast})\cap I(d)} S_{r}^{\perp}.
    \end{equation}
\end{proposition}
\begin{proof}
We denote $S(x^{\ast}) := \bigcup\limits_{r\in I(x^{\ast})}S_{r}$. Combing (\ref{eqn-3.4}) with (\ref{eqn-3.5}), we have
\begin{equation*}
    \hat{N}_{T_{S}(x^{\ast})}(d) = \hat{N}_{S(x^{\ast})}(d) = \bigcap\limits_{r\in I(x^{\ast})\cap I(d)} {N}_{S_{r}}(d) = \bigcap\limits_{r\in I(x^{\ast})\cap I(d)} S_{r}^{\perp},
\end{equation*}
which completes the proof.
\end{proof}

\par Recently a directional version of the limiting normal cone has been introduced by Ginchev and Mordukhovich \cite{Ginchev-2010}.

\begin{definition}\label{def-2.6}
    \emph{(Directional normal cones) \cite[Definition 2.3]{Ginchev-2010}} Given a closed set $\Omega\subseteq \mathbb{R}^n$, $x^{\ast}\in \Omega$ and $d\in \mathbb{R}^n$. The limiting normal cone to $\Omega$ at $x^{\ast}$ in direction $d$ is defined by
    \begin{equation*}
        N_{\Omega}(x^{\ast};d) := \left \{ v\in \mathbb{R}^{n} \: \vert \: \exists t_{k}\downarrow 0, d^{k}\to d, v^{k}\to v, \:\emph{s.t.}\: v^{k}\in \hat{N}_{\Omega}(x^{k}+t_{k}d^{k}) \right \}.
    \end{equation*}
\end{definition}
\noindent 
From the definition we have $N_{\Omega}(x^{\ast};0)=N_{\Omega}(x^{\ast})$, $N_{\Omega}(x^{\ast};d)\subseteq N_{\Omega}(x^{\ast})$ and $N_{\Omega}(x^{\ast};d) = \emptyset$ if $d\notin T_{\Omega}(x^{\ast})$.

If $C$ is a disjunctive set which is the union of finitely many convex polyhedral sets, by \cite[Lemma 2.1]{Gfre-2014} we have the inclusion
\begin{equation}
    N_{C}(x^{\ast};d)\subseteq \{ v\in N_{C}(x^{\ast}) \:\vert\: v^{T}d=0 \}.\label{disjunctive}
\end{equation}
In the following proposition, we will show that the above inclusion holds as an equality  if $C=S$,  the union of finitely many subspaces.

\begin{proposition}\label{thm-3.4}
    \emph{(Directional normal cone)} Let $x^{\ast}\in S$ and $d\in T_{S}(x^{\ast})$. Then the directional normal cone to $S$ at $x^{\ast}$ in direction $d$ is calculated as
    \begin{equation}
        N_{S}(x^{\ast};d) = \{ v\in N_{S}(x^{\ast}) \:\vert\: v^{T}d=0 \} = \bigcup\limits_{r\in I(x^{\ast})\cap I(d)} S_{r}^{\perp}.
    \end{equation}
\end{proposition}
\begin{proof}

From (\ref{eqn-3.6}), we have
\begin{equation}\label{eqn-3.4.1}
    \{ v\in N_{S}(x^{\ast}) \:\vert\: v^{T}d=0 \} = \bigcup\limits_{r\in I(x^{\ast})} \{ v\in N_{S_{r}}(x^{\ast}) \:\vert\: v^{T}d=0 \}.
\end{equation}
By \cite[Lemma 2.1]{Gfre-2014}, since $S_r$ is convex,
\begin{equation}\label{eqn-3.4.3}
    \left  \{ v\in {N}_{S_{r}}(x^{\ast}) \:\vert\: v^{T}d=0 \right \} =  {N}_{T_{S_{r}}(x^{\ast})}(d)\qquad \forall {r\in I(x^{\ast})}.
\end{equation}
Further, since $T_{S_{r}}(x^{\ast}) = S_{r}$ for each $r\in I(x^{\ast})$, we have
\begin{equation}\label{eqn-3.4.4}
    \bigcup\limits_{r\in I(x^{\ast})} {N}_{T_{S_{r}}(x^{\ast})}(d) = \bigcup\limits_{r\in I(x^{\ast})} {N}_{S_{r}}(d) = \bigcup\limits_{r\in I(x^{\ast})\cap I(d)} S_{r}^{\perp}.
\end{equation}
Combing (\ref{disjunctive}), (\ref{eqn-3.4.1}),
 (\ref{eqn-3.4.3}) and (\ref{eqn-3.4.4}), we establish the inclusion
\begin{equation*}
   N_{S}(x^{\ast};d)\subseteq  \{ v\in N_{S}(x^{\ast}) \:\vert\: v^{T}d=0 \} = \bigcup\limits_{r\in I(x^{\ast})\cap I(d)} S_{r}^{\perp}.
\end{equation*}

Reversely, for any $v\in \bigcup\limits_{r\in I(x^{\ast})\cap I(d)} S_{r}^{\perp}$, there exists $r'\in I(x^{\ast})\cap I(d)$ such that $v\in S_{r'}^{\perp}$. Since there are only finitely many subspaces, we can find a unit vector $u'$ such that $u'\in S_{r'}$ but $u'\notin S_{j}$ for other $j\neq r'$. We define the sequences $\{d^{k}\}$ as
$
    d^{k}:=d+\frac{1}{k}u'.
$
Then $d^{k}\to d$ as $k\to \infty$. 
Since $r'\in I(x^{\ast})\cap I(d)$ we have $x^{\ast}\in S_{r'}$ and $d\in S_{r'}$, which implies that $d^{k}\in S_{r'}$ but $d^{k}\notin S_{j}$ for other $j\neq r'$. For any sequence $\{t_{k}\}$ such that $t_{k}\to 0^{+}$ as $k\to \infty$, it follows that $x^{\ast}+t_{k}d^{k}\in S_{r'}$ but $d^{k}\notin S_{j}$ for other $j\neq r'$. Therefore, we have
\begin{equation*}
    v\in S_{r'}^{\perp} = {N}_{S_{r'}}(x^{\ast}+t_{k}d^{k}) = \hat{N}_{S}(x^{\ast}+t_{k}d^{k}),
\end{equation*}
where the last equation is from (\ref{eqn-3.5}). From the definition of the directional normal cone in Definition \ref{def-2.6}, we have
$
    v\in N_{S}(x^{\ast};d),
$
which implies
 $   \bigcup\limits_{r\in I(x^{\ast})\cap I(d)} S_{r}^{\perp} \subseteq N_{S}(x^{\ast};d)
$
as desired.
\end{proof}


\subsection{The generator sets for disjunctive set of subspaces}


    Let $A$ be a set with finitely many linearly independent vectors and $D$ be a subspace. We say $A$ is the generator of $D$ if
    $    D = \mathcal{G}(A) := \emph{span}(A).$

Now we describe the regular normal cone to the  disjunctive set of subspaces by its generator set. Without loss of generality, we assume that each $S_r$ is represented by 
\begin{equation}\label{eqn-3.3}
    S_{r} := \{ x\in \mathbb{R}^n \:\vert\: \langle a^{r}_{j},x \rangle = 0, j\in \mathcal{E}_{r} \},
\end{equation}
where   $\mathcal{E}_{r}$ is a finite index set and vectors $\{a_{j}^{r}\}_{j\in \mathcal{E}_{r}}$ are linearly independent. For $x^{\ast}\in S_{r}$, let $$A_{S_{r}}:= \{a^{r}_{j}\:\vert\:j\in \mathcal{E}_{r}\}$$ denote the generator  of $S_{r}$. Then
  $  N_{S_{r}}(x^{\ast}) = \mathcal{G}(A_{S_{r}}).$

By (\ref{eqn-3.5}), the regular normal cone to set $S$ at $x^{\ast}\in S$  is the intersection of finitely many subspaces, hence it is still a subspace. So we may assume that  $\hat{N}_{S}(x^{\ast})$ is generated by a set of linearly independent vectors denoted by $\hat{A}_{S}(x^{\ast})$. That is, we have
\begin{equation}\label{eqn-3.24}
    \hat{N}_{S}(x^{\ast}) = \mathcal{G} (\hat{A}_{S}(x^{\ast})).
\end{equation}
We call $\hat{A}_{S}(x^{\ast})$ satisfying (\ref{eqn-3.24}) the generator set of the regular normal cone to $S$ at $x^{\ast}$. 

By (\ref{eqn-3.9}) and (\ref{eqn-3.24}) we have that for some $\delta>0$, 
\begin{equation}\label{eqn-3.25}
    N_{S}(x^{\ast}) = \bigcup\limits_{x\in \mathbb{B}_{\delta}(x^{\ast})}  \hat{N}_{S}(x) = \bigcup\limits_{x\in \mathbb{B}_{\delta}(x^{\ast})} \mathcal{G} (\hat{A}_{S}(x)).
\end{equation}
We also define the set
\begin{equation}\label{eqn-3.26}
    A_{S}(x^{\ast}) := \bigcup\limits_{x\in \mathbb{B}_{\delta}(x^{\ast})} \hat{A}_{S}(x),
\end{equation}
where $\delta>0$ is the constant satisfying condition (\ref{eqn-3.25}). 

By (\ref{eqn-3.6}), we know that
\begin{equation}
    N_{S}(x^{\ast}) = \bigcup\limits_{r\in I(x^{\ast})} N_{S_{r}}(x^{\ast}).
\end{equation}
Therefore an interesting question is, what is the relationship between generator set $A_{S}(x^{\ast})$ and the union of all $A_{S_{r}}$ where $r\in I(x^{\ast})$.   While it is shown that the inclusion $\subseteq$ holds if each $S_r\subseteq \mathbb{R}^d$  with $d=1, 2$ is convex polyhedral by  Lemma 2.1 in \cite{Xu-2021}, in the following lemma, we show that the equality holds if each $S_r$ is a subspace.  

\begin{lemma}\label{lem-3.1}
    Let $S = \bigcup\limits_{r=1}^{R} S_{r} \subseteq \mathbb{R}^{q}$ where $S_{r}$ is a subspace. Then for any $x^{\ast}\in S$, we have
    \begin{equation}\label{Lemmaeqn}
        A_{S}(x^{\ast}) = \bigcup\limits_{r\in I(x^{\ast})} A_{S_{r}}.
    \end{equation}
\end{lemma}
\begin{proof} 
By  the definition,
$
    A_{S}(x^{\ast}) = \bigcup\limits_{x\in \mathbb{B}_{\delta}(x^{\ast})} \hat{A}_{S}(x).
    $
By (\ref{eqn-3.5}), we have
\begin{equation}\label{eqn-3.29}
    \hat{N}_{S}(x) = \bigcap\limits_{r\in I(x)} {N}_{S_{r}}(x) \qquad \forall x\in S,
\end{equation}
 which implies the inclusion
\begin{equation}\label{eqn-3.30}
    \hat{A}_{S}(x) \subseteq \bigcup\limits_{r\in I(x)} A_{S_{r}}\qquad \forall x\in S.
\end{equation}
For any $v\in A_S(x^*)=\bigcup\limits_{x\in \mathbb{B}_{\delta}(x^{\ast})} \hat{A}_{S}(x)$, there exists $x'\in \mathbb{B}_{\delta}(x^{\ast})$ such that
\begin{equation}\label{eqn-3.31}
    v\in \hat{A}_{S}(x')\subseteq \bigcup\limits_{r\in I(x')} A_{S_{r}}\subseteq \bigcup\limits_{r\in I(x^{\ast})} A_{S_{r}}.
\end{equation}
Indeed, the first inclusion in (\ref{eqn-3.31}) is from (\ref{eqn-3.30}) directly. It is not difficult to see that we can take $\delta>0$ sufficiently small to guarantee
  $  I(x')\subseteq I(x^{\ast}).$
Hence we justify the second inclusion in (\ref{eqn-3.31}). Therefore,
\begin{equation}\label{eqn-3.34}
   A_{S}(x^{\ast}) =  \bigcup\limits_{x\in \mathbb{B}_{\delta}(x^{\ast})} \hat{A}_{S}(x) \subseteq \bigcup\limits_{r\in I(x^{\ast})} A_{S_{r}}.
\end{equation}
Now  we prove the reverse inclusion. Take any $v\in \bigcup\limits_{r\in I(x^{\ast})} A_{S_{r}}$. Then there exists $r\in I(x^{\ast})$ such that
$
    v\in A_{S_{r}}.
$
Since the set $S$ is the union of finitely many subspaces, we can pick a unit vector $u_{r}$ such that $u_{r}\in S_{r}$ but $u_{r}\notin S_{j}$ for other $j\neq r$. In this way, for any $0<t<\delta$ we know that the point $x_{r} := x^{\ast} + tu_{r}\in \mathbb{B}_{\delta}(x^{\ast})$ satisfies $x_{r}\in S_{r}$ but $x_{r}\notin S_{j}$ for other $j\neq r$. Therefore, it follows from (\ref{eqn-3.5}) that
$
    \hat{N}_{S}(x_{r}) = {N}_{S_{r}}(x_{r}),
$
which implies that
   $ \hat{A}_{S}(x_{r}) = A_{S_{r}}$ and
hence we have
\begin{equation*}
    v\in A_{S_{r}}= \hat{A}_{S}(x_{r}) \subseteq \bigcup\limits_{x\in \mathbb{B}_{\delta}(x^{\ast})} \hat{A}_{S}(x).
\end{equation*}
Then, we have
\begin{equation}\label{eqn-3.36}
    \bigcup\limits_{r\in I(x^{\ast})} A_{S_{r}} \subseteq \bigcup\limits_{x\in \mathbb{B}_{\delta}(x^{\ast})} \hat{A}_{S}(x)=A_S(x^*),
\end{equation}
and hence the proof is complete.
\end{proof}

\section{Optimality conditions and constraint qualifications for MPDSC}\label{section-3}
In this section we study optimality conditions for MPDSC 
 of the form:
\begin{equation}\label{MPDSC}
    \begin{aligned}
         \min_{x}\quad & f(x)\\
         \text{s.t.}\quad & g(x)\leq 0, h(x) = 0, \Phi_{i}(x)\in S,\: i = 1,\dots,l,
    \end{aligned}
\end{equation}
where $f:\mathbb{R}^n \to \mathbb{R}$, $g:\mathbb{R}^n \to \mathbb{R}^m$, $h:\mathbb{R}^n \to \mathbb{R}^p$, $\Phi_{i}:\mathbb{R}^n \to \mathbb{R}^q$, and constraint set $S := \bigcup_{r=1}^{R} S_{r} \subseteq \mathbb{R}^q$ with $S_{r}$ being a subspace, $r=1,\dots,R$. Unless otherwise mentioned, we assume that all functions are smooth.  We denote the feasible region of problem (\ref{MPDSC}) by $\mathcal{F}$ and the active set for inequality constraints at $x^{\ast}$ by $\mathcal{I}_{g}(x^{\ast}):=\{i\in \{1,\dots,m\} \:\vert\: g_{i}(x^{\ast}) = 0\}$. Problem (\ref{MPDSC}) is more general than what we will need to study CCOP in the form (\ref{pro-4.1}). We could also study CCOP in the more general form as in (\ref{MPDSC}). But for simplicity in comparison with other results for CCOP, we do not state  CCOP in the more general form.
\subsection{Optimality conditions for MPDSC}
According to  the popular terminology in the mathematical program with equilibrium constraints (MPEC), one normally associate the strong (S-) stationarity with the regular normal cone and the Mordukhovich (M-) stationarity with the limiting normal cone (see e.g. \cite[Definition 1]{Fleg-2007}).
Hence if $x^{\ast}$ is a feasible point of MPDSC
we  say that  $x^{\ast}$ is S-/M-stationary   if there exists $\lambda = (\lambda^{g},\lambda^{h},\lambda^{\Phi})$ with $ \lambda^{g}_{i}\geq 0,\:\forall i\in \mathcal{I}_{g}(x^*)$ satisfying
\begin{equation}\label{eqn-18new}
        \nabla f(x^{\ast}) + \sum\limits_{i\in \mathcal{I}_{g}(x^*)}\lambda^{g}_{i} \nabla g_{i}(x^{\ast}) + \sum\limits_{i=1}^{p}\lambda^{h}_{i}\nabla h_{i}(x^{\ast}) + \sum\limits_{i=1}^{l} \nabla \Phi_i(x^*)^T\lambda^\Phi_{i} = 0,
\end{equation}
such that 
\begin{eqnarray}
    \lambda^{\Phi}_{i} & \in &  \hat{N}_{S}(\Phi_{i}(x^{\ast}))\equiv\bigcap_{r\in I(\Phi_{i}(x^{\ast}))}S_{r}^{\perp},\:\forall i=1,\dots,l;\\
     \lambda^{\Phi}_{i} &\in & N_{S}(\Phi_{i}(x^{\ast}))\equiv\bigcup_{r\in I(\Phi_{i}(x^{\ast}))}S_{r}^{\perp},\:\forall i=1,\dots,l,
\end{eqnarray} respectively.

Recently the directional versions of S- and M-stationarity for MPDC have been introduced and studied by Gfrerer \cite{Gfre-2014}. We denote   the linearization cone  of MPDSC (\ref{MPDSC}) at $x^*$ as
\begin{equation*}
    L_{\mathcal{F}}^{lin}(x^{\ast})=\left\{d \in \mathbb{R}^{n} \middle\vert \begin{array}{cl}
    \nabla g_{i}(x^{\ast})d \leq 0, & i \in \mathcal{I}_{g}(x^{\ast}) \\
    \nabla h_{i}({x}^\ast)d = 0, & i \in\{1, \ldots, p\} \\
    \nabla \Phi_{i}(x^{\ast})d\in \bigcup\limits_{r\in I(\Phi_{i}(x^{\ast}))} S_{r}, & i \in\{1, \ldots, l\}
    \end{array}\right\},
\end{equation*} and the critical cone at $x^*$ as $$\mathcal{C}(x^{\ast}) :=\left  \{d\in L^{lin}_{\mathcal{F}}(x^{\ast})\:\vert\:\nabla f(x^{\ast})d\leq 0 \right \}.$$ We also denote the active set for inequality constraints at $x^{\ast}$ in direction $d$ by 
\begin{equation}\mathcal{I}_{g}^*(d):=\{i\in \mathcal{I}_{g}(x^{\ast}) \:\vert\: \nabla g_{i}(x^{\ast})d=0\}.\label{activeset}
\end{equation}
According to \cite[Propositions 3.8 and 3.6]{liang2021ye}, we may define the directional S-/M-stationarity as follows.
\begin{definition}\label{def-3.9}
Let $x^{\ast}$ be a feasible point of MPDSC \emph{(\ref{MPDSC})} and $d\in \mathcal{C}(x^{\ast})$. 
We say $x^{\ast}$ is S-/M-stationary $(d)$ if there exists $\lambda = (\lambda^{g},\lambda^{h},\lambda^{\Phi})$ with $\lambda^{g}_{i}\geq 0,\: \forall i\in \mathcal{I}_{g}^*(d)$  satisfying 
\begin{equation}\label{eqn-18}
        \nabla f(x^{\ast}) + \sum\limits_{i\in \mathcal{I}_{g}^*(d)}\lambda^{g}_{i} \nabla g_{i}(x^{\ast}) + \sum\limits_{i=1}^{p}\lambda^{h}_{i}\nabla h_{i}(x^{\ast}) + \sum\limits_{i=1}^{l}\nabla \Phi_i(x^*)^T\lambda^\Phi_{i} = 0,
\end{equation}
such that
\begin{eqnarray*}
    \lambda^{\Phi}_{i} & \in &  \hat{N}_{T_{S}(\Phi_{i}(x^{\ast}))}(\nabla \Phi_{i}(x^{\ast})d)\equiv\bigcap_{r\in I(\Phi_{i}(x^{\ast}))\cap I(\nabla \Phi_{i}(x^{\ast})d)}S_{r}^{\perp},\:\forall i=1,\dots,l;\\
     \lambda^{\Phi}_{i} &\in & N_{S}(\Phi_{i}(x^{\ast});\nabla \Phi_{i}(x^{\ast})d)\equiv\bigcup_{r\in I(\Phi_{i}(x^{\ast}))\cap I(\nabla \Phi_{i}(x^{\ast})d)}S_{r}^{\perp},\:\forall i=1,\dots,l,
\end{eqnarray*}
respectively. We call the above multiplier $\lambda$ the S-/M- multiplier in direction d respectively.
\end{definition}
In fact, if we take direction $d=0$ in Definition \ref{def-3.9}, then since $\mathcal{I}_{g}^*(0)=\mathcal{I}_{g}(x^*)$,$I(0)=\{1,\dots, R\}$,  we will recover S-/M-stationarity  for MPDSC (\ref{MPDSC}).

We now discuss under what conditions a local minimizer are S-stationary $(d)$ or M-stationary $(d)$, respectively. When $d\not =0$, it is easy to see that the M- stationarity (d) are stronger than the standard M- stationarity while S- stationarity  are stronger than the standard S- stationarity (d).

First we recall the following well-known condition.
\begin{definition}[Metric subregularity constraint qualification/error bound property for MPDSC](\cite[Definition 2.1]{Gfre-2013})
     Given $x^{\ast}\in \mathcal{F}$ and $d\in \mathbb{R}^n$, we say that the metric subregularity constraint qualification (MSCQ) holds at $x^{\ast}$ in direction $d$ for MPDSC  if there exists $\alpha\geq 0$ and $\rho>0$ such that
    \begin{equation*}
        d_{\mathcal{F}}(x)\leq \alpha \Bigl( \sum\limits_{i=1}^{m} \max\{g_{i}(x),0\} + \sum\limits_{i=1}^{p} |h_{i}(x)| + \sum\limits_{i=1}^{l} d_{S}(\Phi_{i}(x)) \Bigr),\: \quad \forall x\in x^{\ast}+V_{\rho,\delta}(d)
    \end{equation*} 
    where 
    $  V_{\rho,\delta}(d):=\Big\{ z\in \rho \mathbb{B}\:\Big\vert\: \big\|\|d\| {z}-\|z\|d \big\|\leq \delta \|z\| \|d\| \Big\}$
    is the so-called directional neighborhood in direction $d$. If $d=0$ in the above, we say that the metric subregularity constraint qualification holds at $x^*$.
\end{definition}


\begin{remark}
Since the metric subregularity is an abstract condition, in practice, one needs to use some verifiable sufficient conditions to ensure its validity.  There are some verifiable sufficient conditions in the literature such as the (directional) quasi-/pseudo normality  \emph{\cite[Definition 4.1]{bai2019directional}} and  the  first/second-order sufficient condition for metric subregularity  \emph{\cite[Theorem 4.3]{Gfre-2013}}.
\end{remark}

We now give the first/second-order sufficient condition for metric subregularity introduced in  \cite[Theorem 4.3]{Gfre-2013} and the directional MPDC-LICQ \cite[Definition 3.6]{Gfre-2014}
 to our problem setting.
\begin{definition}\label{Defn3.3} Let $x^*\in {\cal F}$. 
\begin{itemize}
\item[\emph{(a)}]  We say that the MPDSC first-order sufficient condition for metric subregularity (MPDSC-FOSCMS) holds at $x^*$  in direction $d\in L_{\mathcal{F}}^{lin}(x^{\ast})$ if there is  no nonzero vector $ (\lambda^g, \lambda^h,$ $\lambda^{\Phi}) \in \mathbb{R}^{m}\times \mathbb{R}^n\times \prod_{i=1}^{l}\mathbb{R}^q$  with $\lambda^{g}_{i}\geq 0,\: \forall i\in \mathcal{I}_{g}^*(d)$ and   $\lambda^{\Phi}_{i}\in \bigcup\limits_{r\in I(\Phi_{i}(x^{\ast}))\cap I(\nabla \Phi_{i}(x^{\ast})d)}S_{r}^{\perp}$ satisfying 
\begin{equation} \sum\limits_{i\in \mathcal{I}_{g}^*(d)}\lambda^{g}_{i} \nabla g_{i}(x^{\ast}) + \sum\limits_{i=1}^{p}\lambda^{h}_{i}\nabla h_{i}(x^{\ast}) + \sum\limits_{i=1}^{l}\nabla \Phi_i(x^*)^T\lambda^\Phi_{i} = 0.\label{FOSCMS}
\end{equation}
\item[\emph{(b)}] Suppose that all functions are twice continuously differentiable. We say that the MPDSC second-order sufficient condition for metric subregularity (MPDSC-SOSCMS) holds at $x^*$  in direction $d\in L_{\mathcal{F}}^{lin}(x^{\ast})$ if there is  no nonzero vector $ (\lambda^g, \lambda^h,$ $\lambda^{\Phi}) \in \mathbb{R}^{m}\times \mathbb{R}^n\times \prod_{i=1}^{l}\mathbb{R}^q$  with $\lambda^{g}_{i}\geq 0,\: \forall i\in \mathcal{I}_{g}^*(d)$ and   $\lambda^{\Phi}_{i}\in \bigcup\limits_{r\in I(\Phi_{i}(x^{\ast}))\cap I(\nabla \Phi_{i}(x^{\ast})d)}S_{r}^{\perp}$ satisfying (\ref{FOSCMS}) and the following second-order condition
$$ d^{T} \nabla^{2}_{x} \mathcal{L}^0(x^{\ast},\lambda^{g},\lambda^{h},\lambda^{\Phi}) d\geq 0,$$
where 
$$  \mathcal{L}^0(x,\lambda^{g},\lambda^{h},\lambda^{\Phi}) :=  \langle \lambda^{g},g(x) \rangle + \langle \lambda^{h},h(x) \rangle + \langle \lambda^{\Phi},\Phi(x) \rangle.$$
\item[\emph{(c)}] We say that MPDSC linear independence constraint qualification (MPDSC-LICQ) holds at $x^*$  in direction $d\in L_{\mathcal{F}}^{lin}(x^{\ast})$  if there is no nonzero vector $ (\lambda^g, \lambda^h, \lambda^{\Phi})\in \mathbb{R}^{m}\times \mathbb{R}^n\times \prod_{i=1}^{l}\mathbb{R}^q$ with $ \lambda^{\Phi}_{i}\in \sum\limits_{r\in I(\Phi_{i}(x^{\ast}))\cap I(\nabla \Phi_{i}(x^{\ast})d)} S_{r}^{\perp}$ such that
   (\ref{FOSCMS}) holds.
\end{itemize} 
\end{definition}
When $d=0$, MPDSC-FOSCMS (d) reduces to MPDSC-NNAMCQ, MPDSC-LICQ (d) reduces to MPDSC-LICQ defined in Definition \ref{def-3.12} repsectively. It is obvious that MPDSC-FOSCMS (d) and MPDSC-LICQ (d) with $d\not =0$ are weaker than NNAMCQ and MPDSC-LICQ, respectively.

We denote the Lagrangian function of problem (\ref{MPDSC}) by
\begin{equation*}
    \mathcal{L}(x,\lambda^{g},\lambda^{h},\lambda^{\Phi}) := f(x) + \langle \lambda^{g},g(x) \rangle + \langle \lambda^{h},h(x) \rangle + \langle \lambda^{\Phi},\Phi(x) \rangle.
\end{equation*}

Similarly as  in \cite[Section 4]{liang2021ye}, the following theorem can be obtained directly from the optimality conditions for the disjunctive program originally obtained by Gfrerer in \cite[Theorems 3.3 and 3.17, Corollary 3.20]{Gfre-2014} and summarized in \cite[Theorem 6.1]{Benk-2022} with the calculus rules of Cartisian product rules of tangent and directional normal cones in  \cite[Proposition 3.3]{ye2018verifiable}.

\begin{theorem}\label{thm-3.4.}
    Let $x^{\ast}$ be a local optimal solution of MPDSC \emph{(\ref{MPDSC})} and $d\in \mathcal{C}(x^{\ast})$. 
    \begin{itemize}
   \item[\emph{(i)}] 
If MSCQ in direction $d$ holds at $x^*$, then $x^*$ must be M-stationary (d). Moreover if all functions are twice continuously differentiable, then there exists an M-stationary $(d)$ multiplier $ (\lambda^{g}, \lambda^{h}, \lambda^{\Phi})$  such that the second-order necessary optimality condition holds:
    \begin{equation}\label{second-order}
        d^{T} \nabla^{2}_{x} \mathcal{L}(x^{\ast},\lambda^{g},\lambda^{h},\lambda^{\Phi}) d\geq 0.
    \end{equation}

    \item[\emph{(ii)}] 
  If MPDSC-LICQ $(d)$ is fulfilled at $x^{\ast}$,
then $x^*$ must be S-stationary (d). Moreover  if all functions are twice continuously differentiable, then there exists an S-stationary $(d)$ multiplier $(\lambda^{g}, \lambda^{h}, \lambda^{\Phi})$  such that the second-order necessary optimality condition \emph{(\ref{second-order})} holds.
    \end{itemize}
    Conversely, let $x^{\ast}$ be a feasible solution of MPDSC \emph{(\ref{MPDSC})}. Suppose that all functions are twice continuously differentiable and for every nonzero critical direction $0\neq d\in \mathcal{C}(x^{\ast})$ there exists an S-stationary $(d)$ multiplier $ (\lambda^{g}, \lambda^{h}, \lambda^{\Phi})$  such that the second-order condition strictly holds:
    \begin{equation*}
        d^{T} \nabla^{2}_{x} \mathcal{L}(x^{\ast},\lambda^{g},\lambda^{h},\lambda^{\Phi}) d > 0.
    \end{equation*}
    Then
$x^{\ast}$ is a strict local minimizer of MPDSC \emph{(\ref{MPDSC})}.
\end{theorem}



\subsection{Constraint qualifications for MPDSC from MPDC}
Now using the formulas for tangent and normal cones in Proposition \ref{thm-3.1}, we obtain some constraint qualifications for MPDSC from the corresponding ones for MPDC (see e.g.,  \cite[Definitions 3.1 and 4.2]{Xu-2021}). 

\begin{definition}\label{def-3.12}
    Let $x^{\ast}\in \mathcal{F}$ be a feasible point for MPDSC \emph{(\ref{MPDSC})}. 
    \begin{enumerate}
         \item  We say that $x^{\ast}$ satisfies MPDSC linear independence constraint qualification (MPDSC-LICQ) if there is no nonzero vector $(\lambda^g, \lambda^h,$ $\lambda^{\Phi}) \in \mathbb{R}^{m}\times \mathbb{R}^n\times \prod_{i=1}^{l}\mathbb{R}^q$ with  $\lambda^{\Phi}_{i}\in \sum\limits_{r\in I(\Phi_{i}(x^{\ast}))}S_{r}^{\perp}$ satisfying 
        \begin{eqnarray}
             && 0 = \sum\limits_{i\in \mathcal{I}_{g}(x^{\ast})} \lambda_{i}^{g} \nabla g_{i}(x^{\ast}) + \sum\limits_{i=1}^{p} \lambda_{i}^{h} \nabla h_{i}(x^{\ast}) + \sum\limits_{i=1}^{l} \nabla \Phi_{i}(x^{\ast})^{T}\lambda^{\Phi}_{i}. \label{MPDSC-WLICQ}
        \end{eqnarray}
        \item  We say that $x^{\ast}$ satisfies MPDSC no nonzero abnormal multiplier constraint qualification (MPDSC-NNAMCQ) if there is no nonzero vector $ (\lambda^g, \lambda^h,$ $\lambda^{\Phi}) \in \mathbb{R}^{m}\times \mathbb{R}^n\times \prod_{i=1}^{l}\mathbb{R}^q$ with $ \lambda^{g}_{i}\geq 0,\:i\in \mathcal{I}_{g}(x^{\ast})$, $\lambda^{\Phi}_{i}\in  \bigcup\limits_{r\in I(\Phi_{i}(x^{\ast}))}S_{r}^{\perp}$ satisfying (\ref{MPDSC-WLICQ}). 
          \item We say that $x^{\ast}$ satisfies MPDSC  relaxed constant positive linear dependence constraint  qualification (MPDSC-RCPLD) if the following conditions hold.
        \begin{itemize}
        \item[\emph{(i)}] The vectors $\{\nabla h_{i}(x)\}^{p}_{i=1}$ have the same rank for all $x \in \mathbb{B}_{\varepsilon}(x^{\ast})$ for some $\varepsilon>0$;
        \item[\emph{(ii)}] Let $J\subseteq \{1,\dots,p\}$ be such that the set of vectors $\{\nabla h_{i}(x^{\ast})\}_{i\in J}$ is a basis for \emph{span}$\{\nabla h_{i}(x^{\ast})\}_{i=1}^{p}$. If there exist index sets $I\subseteq \mathcal{I}_{g}(x^{\ast})$, a nonzero vector $ (\lambda^g, \lambda^h, \lambda^{\Phi})\in \mathbb{R}^{m}\times \mathbb{R}^n\times \prod_{i=1}^{l}\mathbb{R}^q$ with $\lambda^{g}_{i}\geq 0, i\in I$ and $\lambda^{\Phi}_{i}\in  \bigcup\limits_{r\in I(\Phi_{i}(x^{\ast}))}S_{r}^{\perp}$, $i\in \{1,\dots,l\}$ satisfying
         \begin{eqnarray}
             && 0 = \sum\limits_{i\in I} \lambda_{i}^{g} \nabla g_{i}(x^{\ast}) + \sum\limits_{i\in J} \lambda_{i}^{h} \nabla h_{i}(x^{\ast}) + \sum\limits_{i=1}^{l} \nabla \Phi_{i}(x^{\ast})^{T}\lambda^{\Phi}_{i}, \label{MPDSC-WLICQnew}
        \end{eqnarray}
        then the set of vectors
        \begin{equation}\label{eqn-3.38}
            \{\nabla g_{i}(x^{k})\}_{i\in I} \cup \{\nabla h_{i}(x^{k})\}_{i\in J} \cup \bigcup\limits_{\beta_{i}\in A_{i},i\in \{1,\dots,l\}} \{\nabla \Phi_{i}(x^{k})^{T}\beta_{i}\}
        \end{equation}
        is linearly dependent for $k$ sufficiently large, for all sequences $\{x^{k}\}$ satisfying $x^{k}\to x^{\ast}, x^{k}\neq x^{\ast}$ as $k\to \infty$ and any set of linearly independent vectors $A_{i}$ with the following conditions:
        \begin{equation}
       \begin{array}{ll}
       \mbox{ if }   \lambda^{\Phi}_{i} \not =0, &  \mbox{ then }
        A_{i} \subseteq A_{S}(\Phi_{i}(x^{\ast})), \quad  \lambda^{\Phi}_{i} \in \mathcal{G}(A_{i}) \subseteq \bigcup\limits_{r\in I(\Phi_{i}(x^{\ast}))}S_{r}^{\perp};\\
     \mbox{ if }  \lambda^{\Phi}_{i} = 0,   &  \mbox{ then } A_{i} = \emptyset.
     \end{array}\label{generatorset}
        \end{equation} 
        \end{itemize}
         \item  We say that $x^{\ast}$ satisfies MPDSC constant positive linear dependence constraint  qualification (MPDSC-CPLD) if all condition (ii) in  MPDSC-RCPLD hold with the index set $J$ is taken as a arbitrary subset of $\{1,\dots,p\}$.
%
        \item  We say that $x^{\ast}$ satisfies MPDSC constant rank constraint qualification (MPDSC-CRCQ) if for every index sets $I\subseteq \mathcal{I}_{g}(x^{\ast}), J\subseteq \{1,\dots,p\}$, $L\subseteq \{1,\dots,l\}$ and $\lambda^{\Phi}_{i} \in  \bigcup\limits_{r\in I(\Phi_{i}(x^{\ast}))}S_{r}^{\perp}, i\in L$,  the set of vectors
        \begin{equation*}
            \{\nabla g_{i}(x^{\ast})\}_{i\in I} \cup \{\nabla h_{i}(x^{\ast})\}_{i\in J} \cup \bigcup\limits_{\beta_{i}\in A_{i},i\in L} \{\nabla \Phi_{i}(x^{\ast})^{T}\beta_{i}\}
        \end{equation*}
        and the set of vectors
        \begin{equation*}
            \{\nabla g_{i}(x^{k})\}_{i\in I} \cup \{\nabla h_{i}(x^{k})\}_{i\in J} \cup \bigcup\limits_{\beta_{i}\in A_{i},i\in L} \{\nabla \Phi_{i}(x^{k})^{T}\beta_{i}\}
        \end{equation*}
        have the same rank for $k$ sufficiently large, for all sequences $\{x^{k}\}$ satisfying $x^{k}\to x^{\ast}, x^{k}\neq x^{\ast}$ as $k\to \infty$ and any set of linearly independent vectors $A_{i}$ satisfying  condition (\ref{generatorset}).
        
        \item  We say that $x^{\ast}$ satisfies MPDSC relaxed constant rank constraint qualification (MPDSC-RCRCQ) if the index set $J$ is taken as $\{1,\dots,p\}$ in MPDSC-CRCQ.

        \item  We say that $x^{\ast}$ satisfies MPDSC-ERCPLD if the following conditions hold.
        \begin{itemize}
        \item[\emph{(i)}] The vectors $\{\nabla h_{i}(x)\}^{p}_{i=1}$ have the same rank for all $x \in \mathbb{B}_{\varepsilon}(x^{\ast})$ for some $\varepsilon>0$;
        \item[\emph{(ii)}] Let $J\subseteq \{1,\dots,p\}$ be such that the set of vectors $\{\nabla h_{i}(x^{\ast})\}_{i\in J}$ is a basis for \emph{span}$\{\nabla h_{i}(x^{\ast})\}_{i=1}^{p}$. If there exist index sets $I\subseteq \mathcal{I}_{g}(x^{\ast})$ and $L\subseteq \{1,\dots,l\}$, $\lambda^{\Phi}_{i} \in  \bigcup\limits_{r\in I(\Phi_{i}(x^{\ast}))}S_{r}^{\perp}, i\in L$ such that the set of vectors
        \begin{equation*}
            \{\nabla g_{i}(x^{\ast})\}_{i\in I} \cup \Bigl \{ \{\nabla h_{i}(x^{\ast})\}_{i\in J} \cup \bigcup\limits_{\beta_{i}\in A_{i},i\in L} \{\nabla \Phi_{i}(x^{\ast})^{T}\beta_{i}\} \Bigr \}
        \end{equation*}
        is positive linearly dependent, then the set of vectors
        \begin{equation*}
            \{\nabla g_{i}(x^{k})\}_{i\in I} \cup \{\nabla h_{i}(x^{k})\}_{i\in J} \cup \bigcup\limits_{\beta_{i}\in A_{i},i\in L} \{\nabla \Phi_{i}(x^{k})^{T}\beta_{i}\}
        \end{equation*}
        is linearly dependent for $k$ sufficiently large, for all sequences $\{x^{k}\}$ satisfying $x^{k}\to x^{\ast}, x^{k}\neq x^{\ast}$ as $k\to \infty$ and any set of linearly independent vectors $A_{i}$ satisfying condition (\ref{generatorset}).
        \end{itemize}

    \end{enumerate}
\end{definition}


\subsection{RCPLD as a sufficient condition for MSCQ}
Unlike the differentiable nonlinear program,  RCPLD for MPDC is not a sufficient condition for MSCQ in general. In order to obtain a RCPLD type sufficent condition for MSCQ,  Xu and Ye introduced MPDC-piecewise RCPLD and show that it is a sufficient condition for error bounds \cite[Theorem 4.2]{Xu-2021}. It turns out that for MPDSC many constraint qualifications such as CRCQ, RCRCQ, CPLD, ERCPLD, and RCPLD coincide with their piecewise versions, respectively. Here we only focus on the discussions on RCPLD and the piecewise RCPLD and omit  discussions on other constraint qualifications as well as their piecewise versions since the derivation would be similar. Consequently, since RCPLD coincides with piecewise RCPLD,  RCPLD implies error bounds for MPDSC.

First we present a  simple example to illustrate the equivalence of RCPLD and the piecewise RCPLD for MPDSC.

\begin{example}
   Consider the following cardinality constrained system in $\mathbb{R}^3$:
    \begin{equation}\label{eqn-3.45}
        g(x)\leq 0, h(x) = 0, \|x\|_{0}\leq 2.
    \end{equation}
    From \emph{Figure \ref{two examples}} we know that problem \emph{(\ref{eqn-3.45})} is equivalent to
    \begin{equation*}
        g(x)\leq 0, h(x) = 0, \Phi(x) := x\in S,
    \end{equation*}
    with $S = S_{1} \cup S_{2} \cup S_{3}$ such that $S_{1} = \{0\} \times \mathbb{R} \times \mathbb{R}$, $S_{2} = \mathbb{R} \times \{0\} \times \mathbb{R}$ and $S_{3} = \mathbb{R} \times \mathbb{R} \times \{0\}$.
     Assume that point $x^{\ast} = (0,0,1)$ is feasible for system \emph{(\ref{eqn-3.45})}. According to Definition \ref{def-3.12}(7), we say that RCPLD holds at $x^*$ if 
   \begin{itemize}
        \item[\emph{(i)}] The vectors $\{\nabla h_{i}(x)\}^{p}_{i=1}$ have the same rank for all $x \in \mathbb{B}_{\varepsilon}(x^{\ast})$ for some ${\varepsilon}>0$;
        \item[\emph{(ii)}] Let $J\subseteq \{1,\dots,p\}$ be such that the set of vectors $\{\nabla h_{i}(x^{\ast})\}_{i\in J}$ is a basis for \emph{span}$\{\nabla h_{i}(x^{\ast})\}_{i=1}^{p}$. If there exist index an index set  $I\subseteq \mathcal{I}_{g}(x^{\ast})$, a multiplier $\lambda = (\lambda^g, \lambda^h, \lambda^{\Phi})\in \mathbb{R}^{m}\times \mathbb{R}^n\times \mathbb{R}^3$ with $\lambda^{g}_{i}\geq 0, i\in I$ and  $\lambda^{\Phi}\in N_S(x^*)=S_1^\perp \cup S_2^\perp$ such that
        \begin{equation}
            0 = \sum\limits_{i\in I} \lambda_{i}^{g} \nabla g_{i}(x^{\ast}) + \sum\limits_{i\in J} \lambda_{i}^{h} \nabla h_{i}(x^{\ast}) + \lambda^{\Phi},\label{eqn35}
        \end{equation}
        then the set of vectors
        \begin{equation*}
            \{\nabla g_{i}(x^{k})\}_{i\in I} \cup \{\nabla h_{i}(x^{k})\}_{i\in J} \cup A
        \end{equation*}
        is linearly dependent for $k$ sufficiently large, for all sequences $\{x^{k}\}$ satisfying $x^{k}\to x^{\ast}, x^{k}\neq x^{\ast}$ as $k\to \infty$ and any set of linearly independent vectors $A\subseteq A_S(x^*)=\{e_1,e_2\}$ such that 
        $$ \begin{array}{ll}
        \mbox{ if } \lambda^\Phi\not =0 & \mbox{ then } \lambda^\Phi  \in {\cal G}(A)\subseteq S_1^\perp \cup S_2^\perp,\\
        \mbox{ if } \lambda^\Phi =0 & \mbox{ then } A=\emptyset.\end{array} $$ 
        \end{itemize}
If $\lambda^\Phi\not =0$, then since $\lambda^{\Phi}\in S_1^\perp \cup S_2^\perp$, either $\lambda^\Phi \in {\cal G}(e_1) \subseteq  S_1^\perp= \mathbb{R} \times \{0\} \times \{0\}$ or $\lambda^\Phi \in {\cal G}(e_2) \subseteq  S_2^\perp=\{0\} \times \mathbb{R} \times \{0\}$. Hence the set $A$ above must be taken as $e_1$ if $\lambda^\Phi \in  S_1^\perp$ and  $e_2$ if $\lambda^\Phi_2 \in  S_2^\perp$.
    
We now consider the piecewise RCPLD.     There are three partitions of index set $\{1\}$ into sets $P = \{P_{1},P_{2},P_{3}\}$: \emph{(i)} $P_{1} = \{1\}, P_{2} = \emptyset, P_{3} = \emptyset$; \emph{(ii)} $P_{1} = \emptyset, P_{2} = \{1\}, P_{3} = \emptyset$; \emph{(iii)} $P_{1} = \emptyset, P_{2} = \emptyset, P_{3} = \{1\}$. Since $x^{\ast}\notin S_{3}$, we only have two possible subsystems. Therefore, the piecewise RCPLD holds for system \emph{(\ref{eqn-3.45})} at $x^*$ if RCPLD holds for each of the following two subsystems at $x^*$:
    \[ (\text{$\mathcal{P}_{1}$})\left\{ \begin{gathered}
    g(x)\leq 0, \\
    h(x) = 0, \\
    x\in S_{1}.
    \end{gathered} \qquad \qquad \emph{and} \qquad \qquad (\text{$\mathcal{P}_{2}$})\right\{
    \begin{gathered}
    g(x)\leq 0, \\
    h(x) = 0, \\
    x\in S_{2}.
    \end{gathered}
    \]
But RCPLD for subsystem  $\mathcal{P}_{1}$ holding at $x^*$ if   condition (i) above holds and the following conditions hold. Let $J\subseteq \{1,\dots,p\}$ be such that the set of vectors $\{\nabla h_{i}(x^{\ast})\}_{i\in J}$ is a basis for \emph{span}$\{\nabla h_{i}(x^{\ast})\}_{i=1}^{p}$. If there exist  an index set  $I\subseteq \mathcal{I}_{g}(x^{\ast})$, a multiplier $\lambda = (\lambda^g, \lambda^h, \lambda^{\Phi})\in \mathbb{R}^{m}\times \mathbb{R}^n\times \mathbb{R}^3$ with $\lambda^{g}_{i}\geq 0, i\in I$ and  $\lambda^{\Phi}\in S_1^\perp $ such that (\ref{eqn35}) holds,
        then the set of vectors
        $
            \{\nabla g_{i}(x^{k})\}_{i\in I} \cup \{\nabla h_{i}(x^{k})\}_{i\in J} \cup A
       $
        is linearly dependent for $k$ sufficiently large, for all sequences $\{x^{k}\}$ satisfying $x^{k}\to x^{\ast}, x^{k}\neq x^{\ast}$ as $k\to \infty$ and any set of linearly independent vectors $A\subseteq A_{S_1}=\{e_1\}$ such that 
        $$ \begin{array}{ll}
        \mbox{ if } \lambda^\Phi\not =0 & \mbox{ then } \lambda^\Phi  \in {\cal G}(A)\subseteq S_1^\perp ,\\
        \mbox{ if } \lambda^\Phi =0 & \mbox{ then } A=\emptyset.\end{array} $$ 
       RCPLD for subsystem  $\mathcal{P}_{2}$ holding at $x^*$ if the RCPLD for subsystem $\mathcal{P}_{1}$ with $S_1^\perp$ replaced by $S_2^\perp$ and $e_1$ replaced by $e_2$. Hence obviously RCPLD coincides with the piecewise RCPLD.
 \end{example}

 Now we consider the general case and let $x^{\ast}$ be feasible for MPDSC (\ref{MPDSC}). Let sets $P_{1},\dots,P_{R}$ (sometimes some of them may be empty) be a partition of $\{1,\dots,l\}$. We denote such partition by $P:=\{P_{1},\dots,P_{R}\}$ and consider the subsystem for partition $P$:
\begin{equation}\label{pro-3.41}
    \begin{cases}
    g(x)\leq 0,\: h(x) = 0,\\
    \Phi_{i}(x)\in S_{1},\: i\in P_{1},\\
    \quad\quad\quad \vdots\\
    \Phi_{i}(x)\in S_{R},\: i\in P_{R}.
    \end{cases}
\end{equation}
We denote the feasible region of subsystem (\ref{pro-3.41}) by $\mathcal{F}_{P}$. 
Applying the definition for MPDC \cite[Definition 4.1]{Xu-2021} to MPDSC we have the following definition.
\begin{definition}\label{def-3.14}
    \emph{(MPDSC-PRCPLD)}  We say that the piecewise RCPLD holds for MPDSC \emph{(\ref{MPDSC})} at $x^{\ast}\in \mathcal{F}$, if MPDSC-RCPLD holds for subsystem \emph{(\ref{pro-3.41})} for any partition $P = \{P_{1},\dots,P_{R}\}$ such that $x^{\ast}\in \mathcal{F}_{P}$. That is, the following conditions hold for any partition $P = \{P_{1},\dots,P_{R}\}$ such that $x^{\ast}\in \mathcal{F}_{P}$.
    \begin{itemize}
        \item[\emph{(i)}] The vectors $\{\nabla h_{i}(x)\}^{p}_{i=1}$ have the same rank for all $x \in \mathbb{B}_{\varepsilon}(x^{\ast})$ for some $\varepsilon>0$;
        \item[\emph{(ii)}] Let $J\subseteq \{1,\dots,p\}$ be such that the set of vectors $\{\nabla h_{i}(x^{\ast})\}_{i\in J}$ is a basis for \emph{span}$\{\nabla h_{i}(x^{\ast})\}_{i=1}^{p}$. If there exist index sets $I\subseteq \mathcal{I}_{g}(x^{\ast})$, a nonzero vector $ (\lambda^g, \lambda^h, \lambda^{\Phi})^{T}\in \mathbb{R}^{m}\times \mathbb{R}^n\times \prod_{i=1}^{l}\mathbb{R}^q$ with $\lambda^{g}_{i}\geq 0, i\in I$ and $\lambda^{\Phi}_{i}\in  S_{r}^{\perp}$ for $i\in P_{r}$, $r = 1,\dots,R$ such that
        \begin{equation}\label{eqn-3.42}
            0 = \sum\limits_{i\in I} \lambda_{i}^{g} \nabla g_{i}(x^{\ast}) + \sum\limits_{i\in J} \lambda_{i}^{h} \nabla h_{i}(x^{\ast}) + \sum\limits_{i\in P_{1}} \nabla \Phi_{i}(x^{\ast})^{T}\lambda^{\Phi}_{i} + \cdots + \sum\limits_{i\in P_{R}} \nabla \Phi_{i}(x^{\ast})^{T}\lambda^{\Phi}_{i},
        \end{equation}
        then the set of vectors
        \begin{equation}\label{eqn-3.43}
            \{\nabla g_{i}(x^{k})\}_{i\in I} \cup \{\nabla h_{i}(x^{k})\}_{i\in J}
             \bigcup\limits_{\beta^{r}_{i}\in A^{r}_{i},i\in P_{r},r=1,\dots,R} \{\nabla \Phi_{i}(x^{k})^{T} \beta^{r}_{i}\}
        \end{equation}
        is linearly dependent for $k$ sufficiently large, for all sequences $\{x^{k}\}$ satisfying $x^{k}\to x^{\ast}, x^{k}\neq x^{\ast}$ as $k\to \infty$ and any set of linearly independent vectors $A^{r}_{i}$ with $r\in I(\Phi_{i}(x^{\ast}))$ satisfying the following condition
        \begin{equation}\label{eqn-3.44}
   \begin{array}{ll}
   \mbox{ if } \lambda^{\Phi}_{i} \not =0,  &  \mbox{ then }  
             \quad  \lambda^{\Phi}_{i} \in \mathcal{G}(A^{r}_{i}) \subseteq  S_{r}^{\perp},   \quad   A^{r}_{i} \subseteq A_{S_{r}}\\
            \mbox{ if } \lambda^{\Phi}_{i} = 0, &  \mbox{ then } \quad A^{r}_{i} = \emptyset.
            \end{array}
        \end{equation} 
        \end{itemize}
\end{definition}



Now we show that MPDSC-RCPLD coincides with MPDSC-piecewise RCPLD. As we can see as follows, the following two equalities from Proposition \ref{thm-3.1} and Lemma \ref{lem-3.1} play key roles in the proof of Theorem \ref{thm-3.16}:
\begin{align}
    N_{S}(\Phi_{i}(x^{\ast})) &= \bigcup\limits_{r\in I(\Phi_{i}(x^{\ast}))} N_{S_{r}}(\Phi_{i}(x^{\ast})), \label{eqn-3.46}\\
    A_{S}(\Phi_{i}(x^{\ast})) &= \bigcup\limits_{r\in I(\Phi_{i}(x^{\ast}))} A_{S_{r}}. \label{eqn-3.47}
\end{align}
In fact, we can give a more accurate description of generator sets $A_{i},\:i=1,\dots,l$ when we consider MPDSC instead of general MPDC. The following lemma is useful in what follows. 

\begin{lemma}\label{lem-3.2}
   For  $0\neq\lambda\in N_{S}(y)=\bigcup\limits_{r\in I(y)} S_r^\perp$, suppose that  $\lambda$ is generated by a set of linearly independent vectors $A$ which is a subset of the generator set of the limiting normal cone of $S$ at $y$, i.e., 
    \begin{equation*}
         \lambda\in \mathcal{G}(A) \subseteq N_{S}(y)=\bigcup\limits_{r\in I(y)}S_r^\perp,\quad A \subseteq A_{S}(y)=\bigcup\limits_{r\in I(y)} A_{S_{r}}.
   \end{equation*}
   Then if $\lambda \in N_{S_{r}}(y)=S_r^\perp$, $\lambda$ is also generated by a subset of linearly independent vectors $A$ which is a subset of the generator of $S_r$, i.e.,
   \begin{equation}
        \lambda\in \mathcal{G}(A) \subseteq \mathcal{G}(A_{S_{r}})=S_r^\perp,  \quad  A\subseteq A_{S_{r}},
   \label{eqn-3.40}
   \end{equation}
\end{lemma}
\begin{proof}
Suppose on the contrary that  (\ref{eqn-3.40}) does not hold. Without loss of generality, we assume there are $r_{1},r_{2}\in I(y)$ such that
\begin{equation*}
    A\subseteq A_{S_{r_{1}}} \cup A_{S_{r_{2}}}\:\:\textrm{but}\:\:A\nsubseteq A_{S_{r_{1}}}\:\:\textrm{and}\:\:A\nsubseteq A_{S_{r_{2}}}.
\end{equation*}
Since both $S_{r_{1}}$ and $S_{r_{2}}$ are subspaces, it follows that 
\begin{equation*}
    \mathcal{G}(A)\subseteq \textrm{span}\{S_{r_{1}}^{\perp} \cup S_{r_{2}}^{\perp} \}\:\:\textrm{but}\:\:\mathcal{G}(A)\nsubseteq S_{r_{1}}^{\perp}\:\:\textrm{and}\:\:\mathcal{G}(A)\nsubseteq S_{r_{2}}^{\perp},
\end{equation*}
which contradicts the fact that
\begin{equation*}
    \mathcal{G}(A) \subseteq N_{S}(y) = \bigcup\limits_{r\in I(y)} S_{r}^{\perp}.
\end{equation*}
In this way, we complete the proof.
\end{proof}

\begin{theorem}\label{thm-3.16}
    For  \emph{(MPDSC)},  RCPLD coincides with the piecewise RCPLD.
\end{theorem}

\begin{proof}
Without loss of generality assume that the number of subspaces are more than $2$. Condition (i) is identical for both RCPLD and the piecewise RCPLD. In condition (ii), (\ref{MPDSC-WLICQnew}) is also identical to (\ref{eqn-3.42}). Hence (\ref{MPDSC-WLICQnew}) holds for  $\lambda = (\lambda^g, \lambda^h, \lambda^{\Phi})$ with $\lambda^{\Phi}_{i}\in  \bigcup\limits_{r\in I(\Phi_{i}(x^{\ast}))}S_{r}^{\perp}$
if and only if (\ref{eqn-3.42}) for the same $\lambda$ when $\lambda^{\Phi}_{i}\in  S_{r}^{\perp}, r\in I(\Phi_{i}(x^{\ast}))$. The rest of the proof follows from applying Lemma \ref{lem-3.2}.

\end{proof}

The discussions on other constraint qualifications are similar as that of RCPLD, hence we have the following corollary.

\begin{corollary}\label{col-3.1}
    For mathematical programs with disjunctive subspaces constraints \emph{(MPDSC)}, their constraint qualifications such as CRCQ, RCRCQ, CPLD, ERCPLD, and RCPLD coincide with their piecewise versions, respectively.
\end{corollary}

Based on Theorem \ref{thm-3.16}, now we show that for MPDSC the constraint qualification RCPLD implies error bounds for MPDSC.


\begin{theorem}\label{thm-3.17}
    Suppose that MPDSC-RCPLD holds at $x^{\ast}$ which is feasible for problem \emph{(\ref{MPDSC})}. Then, the error bound property holds at $x^{\ast}$.
\end{theorem}
\begin{proof}
The proof is rather straightforward by combining Theorem \ref{thm-3.16} and \cite[Theorem 4.2]{Xu-2021}.
The constraint qualification MPDSC-RCPLD implies the error bound property since MPDSC-RCPLD coincides with MPDSC-piecewise RCPLD.
\end{proof}

We conclude Section \ref{section-3} with Figure \ref{Figure-3.1} which summarizes the relations among various constraint qualifications for MPDSC.

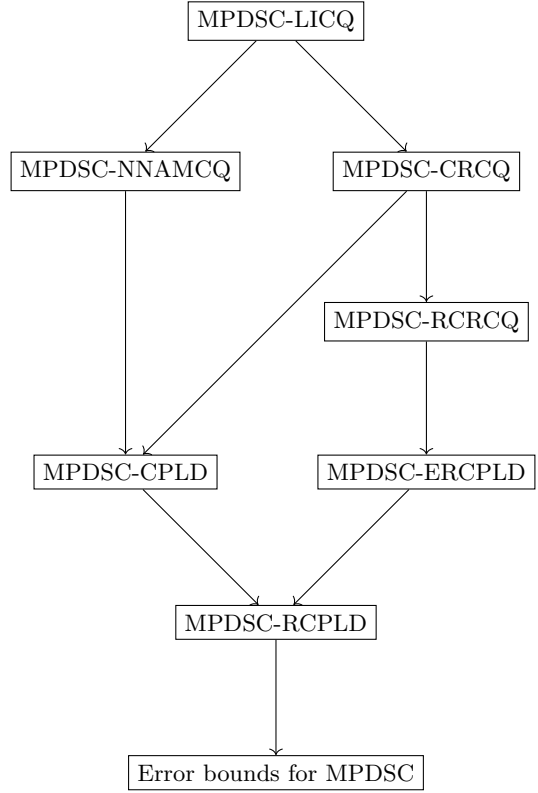
\begin{figure}[t]
    \centering
    \begin{tikzpicture}[scale=2]
        \node[draw = black] (1) at (4,-1) {\small Error bounds for MPDSC};
        \node[draw = black] (2) at (4,0) {\small MPDSC-RCPLD};
        \node[draw = black] (3) at (3,1) {\small MPDSC-CPLD};
        \node[draw = black] (4) at (5,2) {\small MPDSC-RCRCQ};
        \node[draw = black] (5) at (3,3) {\small MPDSC-NNAMCQ};
        \node[draw = black] (6) at (5,3) {\small MPDSC-CRCQ};
        \node[draw = black] (7) at (4,4) {\small MPDSC-LICQ};
        \node[draw = black] (8) at (5,1) {\small MPDSC-ERCPLD};
        \draw[->] (7) -- (5);
        \draw[->] (7) -- (6);
        \draw[->] (5) -- (3);
        \draw[->] (6) -- (3);
        \draw[->] (6) -- (4);
        \draw[->] (3) -- (2);
        \draw[->] (4) -- (8);
        \draw[->] (8) -- (2);
        \draw[->] (2) -- (1);
    \end{tikzpicture}
    \caption{\small{Relations among various constraint qualifications for MPDSC}}
    \label{Figure-3.1}
\end{figure}

\section{Optimality conditions and sufficient condition for error bounds for CCOP}\label{section-4}
In this section we apply main results  obtained for MPDSC to CCOP. Recall that $\mathcal{I}_{s} := \{\mathcal{I}\subseteq \{1,2,\dots,n\}\:\vert\:|\mathcal{I}|=s\}$ and $\mathbb{R}_{\mathcal{I}} := \textrm{span} \{e_{i}\:\vert\:i\in \mathcal{I}\}$. In this section $S:= \bigcup\limits_{\mathcal{I} \in \mathcal{I}_{s}} \mathbb{R}_{\mathcal{I}}$ and  use the notation $I_{\pm}(x) := \{ i\in \{1,\dots,n\} \:\vert\: x_{i}\neq 0\}$ and $I_0(x):= \{ i\in \{1,\dots,n\} \:\vert\: x_{i}=0\}$  for a given vector $x\in \mathbb{R}^{n}$.
The  following proposition   is immediate from Propositions \ref{thm-3.1}-\ref{thm-3.4} taking into account that
$r\in I(x^*) $ here means that $x^*\in \mathbb{R}        _{\cal I}$ and $I_{\pm}(x^{\ast})\subseteq {\cal I}\in {\cal I}_s$.

\begin{proposition}\label{thm-4.1}
Let $x^{\ast}$
and $d\in \mathbb{R}^n$. Then, we have the expressions of various tangent and normal cones  in \emph{Table \ref{Table-4.1}}.
\end{proposition}

\begin{center}
    \renewcommand\arraystretch{1.6}
    \begin{tabular}{|l|c|r|}
    \hline\hline
        Cones & \footnotesize{ $|I_{\pm}(x^{\ast})|<s$\:(or $|I_{\pm}(x^{\ast})\cup I_{\pm}(d)|<s$)} & \footnotesize{$|I_{\pm}(x^{\ast})|=s$\:(or $|I_{\pm}(x^{\ast})\cup I_{\pm}(d)|=s$)} \\ \hline\hline
        $T_{S}(x^{\ast})$ & $\bigcup\limits_{\scriptsize{I_{\pm}(x^{\ast})\subseteq \mathcal{I} \in \mathcal{I}_{s}}} \mathbb{R}_{\mathcal{I}}$ & $\mathbb{R}_{\scriptsize{I_{\pm}(x^{\ast})}}$ \quad\quad\quad\quad\quad \\
        $\hat{N}_{S}(x^{\ast})$ & $\{0\}$ & $\mathbb{R}_{\scriptsize{I_{\pm}(x^{\ast})}}^{\perp}$\quad\quad\quad\quad\quad \\
        $N_{S}(x^{\ast})$ & $\bigcup\limits_{\scriptsize{I_{\pm}(x^{\ast})\subseteq \mathcal{I} \in \mathcal{I}_{s}}} \mathbb{R}_{\mathcal{I}}^{\perp}$ & $\mathbb{R}_{\scriptsize{I_{\pm}(x^{\ast})}}^{\perp}$\quad\quad\quad\quad\quad \\
        $N_{S}(x^{\ast};d)$ & $\bigcup\limits_{\scriptsize{I_{\pm}(x^{\ast})\cup I_{\pm}(d) \subseteq \mathcal{I} \in \mathcal{I}_{s}}} \mathbb{R}_{\mathcal{I}}^{\perp}$ & $\mathbb{R}_{\scriptsize{I_{\pm}(x^{\ast})\cup I_{\pm}(d)}}^{\perp}$\quad\quad\quad\quad \\
        $\hat{N}_{T_{S}(x^{\ast})}(d)$ & $\{0\}$ & $\mathbb{R}_{\scriptsize{I_{\pm}(x^{\ast})\cup I_{\pm}(d)}}^{\perp}$\quad\quad\quad\quad \\ 
    \hline\hline
    \end{tabular}
    \captionof{table}{\small{Various cones to set $S$ in CCOP}}
    \label{Table-4.1}
\end{center}

Now let us write down the  constraint qualifications for CCOP from the ones for MPDSC in Definition \ref{def-3.12}.
In order to do that we need to specify the condition (\ref{generatorset}) to CCOP. For problem CCOP written in the form (\ref{pro-4.1}), $l=1$,  $\Phi_i(x)=x$, $S_r=\mathbb{R}_{\cal I}$ and $\lambda^\Phi_i=\eta$. 
Moreover, $r\in I(\Phi_i(x^*)) $ here means that $x^*\in S_{\cal I}$ and $I_{\pm}(x^{\ast})\subseteq {\cal I}\in {\cal I}_s$.
 Hence the condition (\ref{generatorset}) is
  \begin{equation*}
       \begin{array}{ll}
       \mbox{ if }   \eta \not =0, &  \mbox{ then }
        A \subseteq A_{S}(x^{\ast}), \quad  \eta \in \mathcal{G}(A) \subseteq \bigcup\limits_{I_{\pm}(x^{\ast})\subseteq {\cal I} \in {\cal I}_s} {\mathbb{R}_{\cal I}}^{\perp};\\
     \mbox{ if }  \eta = 0,   &  \mbox{ then } A = \emptyset.
     \end{array}
        \end{equation*} 
Since for  $I_{\pm}(x^{\ast})\subseteq {\cal I}\in \mathcal{I}_{s}$, we have $N_{\mathbb{R}_{\cal I}}(x^{\ast}) = \mathbb{R}_{\cal I}^{\perp}$, 
\begin{equation*}
    A_{\mathbb{R}_{\cal I}} = \{e_{i} \:\vert\: i\in K\:\:\textrm{such that}\:\:K\subseteq I_{0}(x^{\ast})\:\:\textrm{and}\:\:|K|\leq n-s\}.
\end{equation*}
Hence by (\ref{Lemmaeqn}),
$$A_{S}(x^{\ast})=\bigcup\limits_{I_{\pm}(x^{\ast})\subseteq {\cal I} \in {\cal I}_s} A_{\mathbb{R}_{\cal I}}$$
and so in the case where $0\neq \eta\in \mathcal{G}(A)$, by Lemma \ref{lem-3.2} we can take the generator set $A$ above as
\begin{equation*}
    A = \{e_{i} \:\vert\: i\in K\:\:\textrm{such that}\:\:I_{\pm}(\eta)\subseteq K\subseteq I_{0}(x^{\ast})\:\:\textrm{and}\:\:|K|\leq n-s\}.
\end{equation*}

Therefore, we obtain the following definition. Note that the constraint qualification CCOP-LICQ is the same as the one proposed in \cite[Section 5.3]{Mehl-2020-O}.

\begin{definition}\label{Defn4.3}
\emph{(Constraint qualifications for CCOP)} Let $x^{\ast}$ be feasible for CCOP. We say that $x^{\ast}$ satisfies
\begin{enumerate}

    \item \emph{CCOP-LICQ} if
    the family of vectors
      $  \{ \nabla g_{i}(x^{\ast}) \}_{i\in \mathcal{I}_{g}(x^{\ast})} \cup \{ \nabla h_{i}(x^{\ast}) \}_{i=1}^{p} \cup \{ e_{i} \}_{i\in I_{0}(x^{\ast})} 
    $
    is linearly independent;

    \item \emph{CCOP-NNAMCQ} if 
    there exists no nonzero vector $(\lambda^{g},\lambda^{h},\eta)\in \mathbb{R}^m\times \mathbb{R}^p \times \mathbb{R}^n$ with $ \lambda^{g}_{i}\geq 0,\forall i\in \mathcal{I}_{g}(x^*)$ and
     \begin{equation}
        \eta_{i} = 0,\:\forall i\in I_{\pm}(x^{\ast})\:\:\emph{and}\:\:\|\eta\|_{0}\leq n-s.\label{cond}
    \end{equation} 
    satisfying
  \begin{equation} \sum\limits_{i\in \mathcal{I}_{g}(x^*)}\lambda^{g}_{i} \nabla g_{i}(x^{\ast}) + \sum\limits_{i=1}^{p}\lambda^{h}_{i}\nabla h_{i}(x^{\ast}) + \eta = 0;\label{cond1}
  \end{equation}
%
    
      \item \emph{CCOP-RCPLD} if the following conditions hold.
    \begin{itemize}
        \item[\emph{(i)}] The vectors $\{\nabla h_{i}(x)\}^{p}_{i=1}$ have the same rank for all $x \in \mathbb{B}_{\varepsilon}(x^{\ast})$ for some $\varepsilon>0$;
        \item[\emph{(i)}] Let $J\subseteq \{1,\dots,p\}$ be such that the set of vectors $\{\nabla h_{i}(x^{\ast})\}_{i\in J}$ is a basis for \emph{span}$\{\nabla h_{i}(x^{\ast})\}_{i=1}^{p}$.

      If there exist an index set $I\subseteq \mathcal{I}_{g}(x^{\ast})$, a nonzero vector $ (\lambda^{g},\lambda^{h},\eta)\in \mathbb{R}^m\times \mathbb{R}^p \times \mathbb{R}^n$ with $ \lambda^{g}_{i}\geq 0,\forall i\in \mathcal{I}_{g}(x^*)$ and $\eta $ such that  (\ref{cond}) and 
       \begin{equation}
       \sum\limits_{i\in I}\lambda^{g}_{i} \nabla g_{i}(x^{\ast}) + \sum\limits_{i\in J}\lambda^{h}_{i}\nabla h_{i}(x^{\ast}) + \eta = 0,\label{eqn2}
       \end{equation} hold then 
the set of  vectors
        \begin{equation*}
            \{\nabla g_{i}(x^{k})\}_{i\in I} \cup \{\nabla h_{i}(x^{k})\}_{i\in J} \cup \{e_{i}\}_{i\in K}
        \end{equation*}
        is linearly dependent for all sequences $\{x^k\}$ satisfying $x^{k}\to x^{\ast}$, $x^{k}\neq x^{\ast}$ as $k\to \infty$ and all set $K$ satisfying 
    \begin{equation}
 \begin{array}{ll}
      \mbox{    if } \eta\not =0, & \mbox{ then } I_{\pm}(\eta)\subseteq K \subseteq I_{0}(x^{\ast}), \quad
    |K|\leq n-s,\\
     \mbox{    if } \eta = 0, &\mbox{ then } K = \emptyset.
   \end{array}\label{new33}
     \end{equation}
    

    \end{itemize}
 \item \emph{CCOP-CPLD} if all condition (ii) in  CCOP-RCPLD hold with the index set $J$ is taken as a arbitrary subset of $\{1,\dots,p\}$.
 
    \item \emph{CCOP-CRCQ} if for every index sets $I\subseteq \mathcal{I}_{g}(x^{\ast})$, $J\subseteq \{1,\dots,p\}$  and any $\eta$ satifying (\ref{cond})  such that the family of vectors
    \begin{equation*}
        \{ \nabla g(x^{\ast}) \}_{i\in I} \cup \{ \nabla h(x^{\ast}) \}_{i\in J} \cup \{ e_{i} \}_{i\in K}
    \end{equation*}
    and the set of vectors
    \begin{equation*}
        \{ \nabla g(x^{k}) \}_{i\in I} \cup \{ \nabla h(x^{k}) \}_{i\in J} \cup \{ e_{i} \}_{i\in K}
    \end{equation*}
    have the same rank for all sequences $\{x^k\}$ satisfying $x^{k}\to x^{\ast}$, $x^{k}\neq x^{\ast}$ as $k\to \infty$ and any set $K$ satisfying   (\ref{new33}). 
    
    \item \emph{CCOP-RCRCQ} if the index set $J$ is taken as $\{1,\dots,p\}$ in CCOP-CRCQ.

    \item \emph{CCOP-ERCPLD} if the following conditions hold.
    \begin{itemize}
        \item[\emph{(i)}] The vectors $\{\nabla h_{i}(x)\}^{p}_{i=1}$ have the same rank for all $x \in \mathbb{B}_{\varepsilon}(x^{\ast})$ for some $\varepsilon>0$;
        \item[\emph{(ii)}] Let $J\subseteq \{1,\dots,p\}$ be such that the set of vectors $\{\nabla h_{i}(x^{\ast})\}_{i\in J}$ is a basis for \emph{span}$\{\nabla h_{i}(x^{\ast})\}_{i=1}^{p}$. If there exists an index set $I\subseteq \mathcal{I}_{g}(x^{\ast})$, $\eta$ satifying (\ref{cond}) such that the set of vectors
        \begin{equation*}
            \{\nabla g_{i}(x^{\ast})\}_{i\in I} \cup \Bigl \{ \{\nabla h_{i}(x^{\ast})\}_{i\in J} \cup \{e_{i}\}_{i\in K} \Bigr \}
        \end{equation*}
        is positive linearly dependent, then the set of vectors
        \begin{equation*}
            \{\nabla g_{i}(x^{k})\}_{i\in I} \cup \{\nabla h_{i}(x^{k})\}_{i\in J} \cup \{e_{i}\}_{i\in K}
        \end{equation*}
        is linearly dependent for all sequences $\{x^k\}$ satisfying $x^{k}\to x^{\ast}$, $x^{k}\neq x^{\ast}$ as $k\to \infty$, and any set $K$ satisfying  (\ref{new33}). 
    \end{itemize}

\end{enumerate}
\end{definition}

%

%

In Figure \ref{Figure-4.2} we summarize the relations among constraint qualifications for CCOP we discussed above.

\begin{figure}[t]
    \centering
    \begin{tikzpicture}[scale=2]
        \node[draw = black] (1) at (4,-1) {\small Error bounds for CCOP};
        \node[draw = black] (2) at (4,0) {\small CCOP-RCPLD};
        \node[draw = black] (3) at (3,1) {\small CCOP-CPLD};
        \node[draw = black] (4) at (5,2) {\small CCOP-RCRCQ};
        \node[draw = black] (5) at (3,3) {\small CCOP-NNAMCQ};
        \node[draw = black] (6) at (5,3) {\small CCOP-CRCQ};
        \node[draw = black] (7) at (4,4) {\small CCOP-LICQ};
        \node[draw = black] (8) at (5,1) {\small CCOP-ERCPLD};
        \draw[->] (7) -- (5);
        \draw[->] (7) -- (6);
        \draw[->] (5) -- (3);
        \draw[->] (6) -- (3);
        \draw[->] (6) -- (4);
        \draw[->] (3) -- (2);
        \draw[->] (4) -- (8);
        \draw[->] (8) -- (2);
        \draw[->] (2) -- (1);
    \end{tikzpicture}
    \caption{\small{Relations among new constraint qualifications for CCOP}}
    \label{Figure-4.2}
\end{figure}
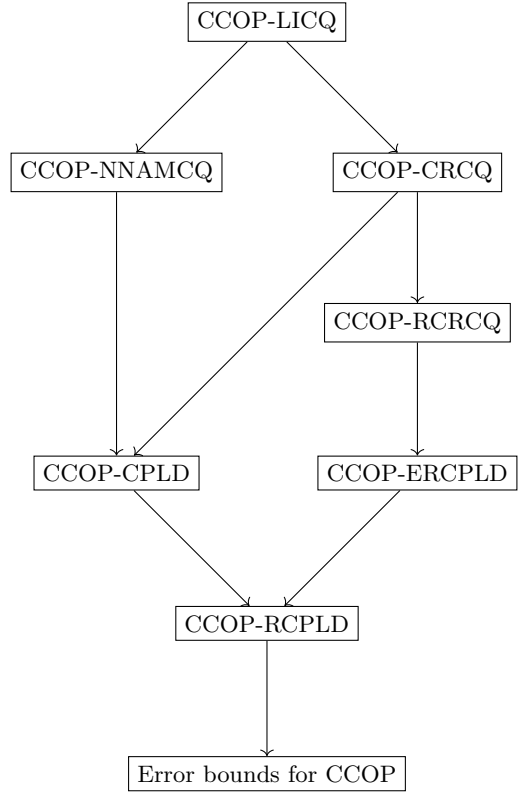

We now study optimality conditions for CCOP.
The critical cone of CCOP at $x^*$ is
\begin{equation*}
   \mathcal{C}(x^{\ast})=\left\{d \in \mathbb{R}^{n} \middle\vert \begin{array}{cl}
    \nabla g_{i}(x^{\ast})d \leq 0, & i \in \mathcal{I}_{g}(x^{\ast}) \\
    \nabla h_{i}(x^*)d = 0, & i \in\{1, \ldots, p\} \\
    \nabla \Phi_{i}(x^{\ast})d\in \bigcup\limits_{\scriptsize{I_{\pm}(x^{\ast})\subseteq \mathcal{I} \in \mathcal{I}_{s}}} \mathbb{R}_{\gamma}, & i \in\{1, \ldots, l\}\\
    \nabla f(x^*)d \leq 0 &
    \end{array}\right\}.
\end{equation*}

Applying Definition \ref{def-3.9} to CCOP (\ref{pro-4.1}), with the help of Proposition \ref{thm-4.1}, we obtain the following directional S-/M-stationary conditions for CCOP. Recall that the active set  $\mathcal{I}_{g}^*(d)$ is defined in (\ref{activeset}).
\begin{definition}\label{def-4.0}
    Let $x^{\ast}\in \mathbb{R}^n$ be a feasible point of CCOP and $d\in \mathcal{C}(x^{\ast})$.
    \begin{itemize}
    \item[\emph{(i)}] We say $x^{\ast}$ is S-stationary in direction $d$ if there exists $(\lambda^{g},\lambda^{h},\eta)\in \mathbb{R}^m\times \mathbb{R}^p\times  \mathbb{R}^n $ with $ \lambda^{g}_{i}\geq 0,\forall i\in \mathcal{I}_{g}^*(d)$ satisfying
    \begin{equation}\label{eqn-4.8}
        \begin{aligned}
            \nabla f(x^{\ast}) + \sum\limits_{i\in \mathcal{I}_{g}^*(d)}\lambda^{g}_{i} \nabla g_{i}(x^{\ast}) &+ \sum\limits_{i=1}^{p}\lambda^{h}_{i}\nabla h_{i}(x^{\ast}) + 
           \eta  = 0,
        \end{aligned}
    \end{equation}
    such that
    \begin{itemize}
        \item[\emph{(a)}] if $|I_{\pm}(x^{\ast}) \cup I_{\pm}(d)| = s$, then $\eta_{i} = 0,\: \forall i\in I_{\pm}(x^{\ast}) \cup I_{\pm}(d)$.
        \item[\emph{(b)}] if $|I_{\pm}(x^{\ast}) \cup I_{\pm}(d)| < s$, then $\eta_{i} = 0,\: \forall i\in \{1,\dots,n\}$.
    \end{itemize}
    
    \item[\emph{(ii)}] We say $x^{\ast}$ is M-stationary $(d)$ if there exists $(\lambda^{g},\lambda^{h},\eta)\in \mathbb{R}^m\times \mathbb{R}^p\times  \mathbb{R}^n $ with $ \lambda^{g}_{i}\geq 0,\forall i\in \mathcal{I}_{g}^*(d)$ satisfying \emph{(\ref{eqn-4.8})} such that
    \begin{equation*}
        \eta_{i} = 0,\:\forall i\in I_{\pm}(x^{\ast})\cup I_{\pm}(d)\:\:\emph{and}\:\:\|\eta\|_{0}\leq n-s.
    \end{equation*}
    \end{itemize}
    Moreover we call $\lambda=(\lambda^g,\lambda^h,\eta)$ satisfying (i) and (ii) S-/M-stationary $(d)$ multiplier respectively.
\end{definition}

Taking direction $d=0$ in Definition \ref{def-4.0}, we have $I(d)=\emptyset$. Hence we recover the S-/M-stationary condition for CCOP which was first proposed by Pan et al. in \cite[Definition 3.1]{pan2017optimality} under the name B-/M-KKT condition. From Definition \ref{def-4.0} it is clear  that the directional M-stationary condition is in general sharper than the standard M-stationary condition while 
the S-stationary condition is in general stronger than the directional S-stationary condition

We now specify the three conditions for MPDSC in  Definition \ref{Defn3.3} to CCOP.

\begin{definition} Let $x^*$ be a feasible solution to CCOP. 
\begin{itemize}
\item[\emph{(a)}]  We say that CCOP-FOSCMS holds at $x^*$  in direction $d\in {\mathcal{C}}(x^{\ast})$ if there exists no nonzero vector  $ (\lambda^{g},\lambda^{h},\eta)\in \mathbb{R}^m\times \mathbb{R}^p \times \mathbb{R}^n $ with $ \lambda^{g}_{i}\geq 0,\forall i\in \mathcal{I}_{g}^*(d)$ and
    $
        \eta_{i} = 0,\:\forall i\in I_{\pm}(x^{\ast})\cup I_{\pm}(d)\:\:\emph{and}\:\:\|\eta\|_{0}\leq n-s
$
    satisfying
   \begin{equation}  \sum\limits_{i\in \mathcal{I}_{g}^*(d)}\lambda^{g}_{i} \nabla g_{i}(x^{\ast}) + \sum\limits_{i=1}^{p}\lambda^{h}_{i}\nabla h_{i}(x^{\ast})+ \eta = 0. 
  \label{FOSC} \end{equation}
\item[\emph{(b)}] Suppose that all functions are twice continuously differentiable. We say that the CCOP-SOSCMS holds at $x^*$  in direction $d\in {\mathcal{C}}(x^{\ast})$ if there is  no nonzero vector $(\lambda^{g},\lambda^{h},\eta)\in \mathbb{R}^m\times \mathbb{R}^p \times \mathbb{R}^n$ with $ \lambda^{g}_{i}\geq 0,\forall i\in \mathcal{I}_{g}^*(d)$ and  $
        \eta_{i} = 0,\:\forall i\in I_{\pm}(x^{\ast})\cup I_{\pm}(d)\:\:\emph{and}\:\:\|\eta\|_{0}\leq n-s
$ satisfying (\ref{FOSC}) and the following second-order condition
$$ d^{T} \nabla^{2}_{x} \mathcal{L}^0(x^{\ast},\lambda^{g},\lambda^{h}) d\geq 0,$$
where 
$\mathcal{L}^0(x,\lambda^{g},\lambda^{h}) :=  \langle \lambda^{g},g(x) \rangle + \langle \lambda^{h},h(x) \rangle.$\item[\emph{(c)}] We say that CCOP-LICQ holds at $x^*$  in direction $d\in {\mathcal{C}}(x^{\ast})$  if
    the family of vectors
      $$  \{ \nabla g_{i}(x^{\ast}) \}_{i\in \mathcal{I}_{g}(x^{\ast})} \cup \{ \nabla h_{i}(x^{\ast}) \}_{i=1}^{p} \cup \{ e_{i} \}_{i\in I_{0}(x^{\ast})\cap I_0(d)} 
    $$
    is linearly independent;
\end{itemize} 
\end{definition}
When $d=0$, since $\mathcal{I}_{g}^*(0)=\mathcal{I}_{g}(x^*)$ CCOP-FOSCMS reduces to CCOP-NNAMCQ, CCOP-LICQ in direction $d$ reduces to CCOP-LICQ defined in Definition \ref{Defn4.3} repsectively. It is obvious that in general we have
\begin{eqnarray}
&& \mbox{CCOP-NNAMCQ}\Longrightarrow\mbox{CCOP-FOSCMS} (d) \Longrightarrow  \mbox{CCOP-SOSCMS} (d) \label{eqn54}\\
&& \mbox{CCOP-LICQ}  \Longrightarrow \mbox{CCOP-LICQ} (d). \label{eqn55}\end{eqnarray}  

We denote the Lagrangian function of CCOP by
\begin{equation*}
    \mathcal{L}(x,\lambda^{g},\lambda^{h}) := f(x) + \langle \lambda^{g},g(x) \rangle + \langle \lambda^{h},h(x) \rangle .
\end{equation*}
The following optimality conditions are the CCOP version of Theorem \ref{thm-3.4.}.
\begin{theorem}
    Let $x^{\ast}$ be a local optimal solution of CCOP. Then the following first and second order necessary optimality conditions hold:
    \begin{itemize}
    \item[\emph{(i)}] Suppose that MSCQ holds at $x^*$ in direction  $d\in \mathcal{C}(x^{\ast})$. Then $x^*$ is M-stationary in direction $d$ and  there exists an M-stationary $(d)$ multiplier $\lambda = (\lambda^{g}, \lambda^{h},\eta)$  such that the second-order condition holds:
    \begin{equation}\label{second-order-new}
    d^T \nabla_x^2 \mathcal{L}(x^*,\lambda^{g},\lambda^{h})d\geq 0
    \end{equation}
    
    \item[\emph{(ii)}] For $d\in \mathcal{C}(x^{\ast})$, assume that CCOP-LICQ $(d)$ is fulfilled at $x^{\ast}$. Then, there exists an S-stationary $(d)$ multiplier $\lambda = (\lambda^{g}, \lambda^{h}, \eta)$  such that the second-order condition \emph{(\ref{second-order-new})} holds.
    \end{itemize}
    Conversely, let $x^{\ast}$ be a feasible solution of CCOP. Suppose that for every nonzero critical direction $0\neq d\in \mathcal{C}(x^{\ast})$ there exists an S-stationary $(d)$ multiplier $\lambda = (\lambda^{g}, \lambda^{h},\eta)$ such that the second-order condition strictly holds:
    \begin{equation*}
      d^T \nabla_x^2 \mathcal{L}(x^*,\lambda^{g},\lambda^{h})d > 0.
    \end{equation*}
    Then $x^{\ast}$ is a strict local minimizer of CCOP \emph{(\ref{pro-4.1})}.
\end{theorem}
Since direction M-stationary condition (d) implies M-stationary condition while S-stationary condition implies S-stationary condition  (d) respectively, we have the following corollary.
\begin{corollary} \label{Cor4.1}Let $x^{\ast}$ be a local optimal solution of CCOP. Then the following first and second order necessary optimality conditions hold:
 Suppose that either CCOP-SOSCMS holds at $x^*$ in direction  $d\in \mathcal{C}(x^{\ast})$ or one of the constraint qualification presented in Figure \ref{Figure-4.2} holds. Then $x^*$ is M-stationary  and  there exists an M-stationary multiplier $\lambda = (\lambda^{g}, \lambda^{h}, \eta)$  such that the second-order condition (\ref{second-order-new}) holds.
    
    Conversely, let $x^{\ast}$ be a feasible solution of CCOP. Suppose that for every nonzero critical direction $0\neq d\in \mathcal{C}(x^{\ast})$ there exists an S-stationary $(d)$ multiplier $\lambda = (\lambda^{g}, \lambda^{h}, \eta)$ associated with $x^{\ast}$ such that the second-order condition (\ref{second-order-new}) strictly holds.
    Then $x^{\ast}$ is a strict local minimizer of CCOP \emph{(\ref{pro-4.1})}.
\end{corollary}
Corollary \ref{Cor4.1} has improved the first order necessary optimality conditions in Pan et al. in \cite[Theorems 3.2]{pan2017optimality}(ii) 
since CCOP-SOSCMS is weaker than  R-MFCQ. It also improved the  second order sufficient in Pan et al. in \cite[Theorems  4.2]{pan2017optimality} since S-stationarity  implies S-stationarity (d).

We now conclude this section with the following application of our error bound results. 
\begin{theorem} Let $x^*\in \mathcal{F}_{\rm CCOP}$ where $\mathcal{F}_{\rm CCOP}$ denotes the  feasible region of  CCOP. If one of the constraint qualification presented in Figure \ref{Figure-4.2} holds, then   there exist $\alpha\geq 0$ and $\rho>0$ such that
    \begin{equation*}
        d_{\mathcal{F}_{\rm CCOP}}(x)\leq \alpha \Bigl( \sum\limits_{i=1}^{m} \max\{g_{i}(x),0\} + \sum\limits_{i=1}^{p} |h_{i}(x)|\Bigr )\qquad 
    \forall    x\in  \mathbb{B}_\rho (x^{\ast})\cap S.
    \end{equation*} 
    Moreover if $x^*$ is a local optimal solution of CCOP then it is also a local optimal solution of the exact penalty problem (\ref{penalty}) for any $\mu \geq L_f \alpha$ where $L_f$ is the Lipschitz constant of $f$ at $x^*$.
    
\end{theorem}
\begin{proof} The weakest constraint qualification in Figure \ref{Figure-4.2} is CCOP-RCPLD. The error bound property holds by Theorem \ref{thm-3.17}. With the error bound property and the continuous differentiability of $f$, the exact penalty result follows by using the Clarke's exact penalty principle; see \cite[Proposition 2.4.3]{Clar-1990} or \cite[Theorem 4.2]{ye}.
\end{proof}

\end{document}